\documentclass[a4]{article}

\usepackage{hyperref}
\usepackage{amsfonts}
\usepackage{amsmath}
\usepackage{amsthm}
\usepackage{cleveref}
\usepackage{tikz-cd}

\title{Categorical approach to graph limits}

\author{
Martin Dole{\v{z}}al\thanks{Supported by EXPRO project 20-31529X awarded by the Czech Science Foundation and by the Institute of Mathematics, Czech Academy of Sciences (RVO 67985840).}\\
\href{mailto:dolezal@math.cas.cz}{dolezal@math.cas.cz}
\and
Wies{\l}aw Kubi{\'s}\thanks{Supported by EXPRO project 20-31529X awarded by the Czech Science Foundation and by the Institute of Mathematics, Czech Academy of Sciences (RVO 67985840).}\\
\href{mailto:kubis@math.cas.cz}{kubis@math.cas.cz}
}

\date{
Institute of Mathematics, Czech Academy of Sciences\\[2ex]
\today
}

\newtheorem{theorem}{Theorem}[section]
\newtheorem{proposition}[theorem]{Proposition}
\newtheorem{lemma}[theorem]{Lemma}
\newtheorem{claim}[theorem]{Claim}
\newtheorem{corollary}[theorem]{Corollary}
\theoremstyle{definition}
\newtheorem{definition}[theorem]{Definition}
\newtheorem{example}[theorem]{Example}
\newtheorem{remark}[theorem]{Remark}

\def\sqg{$\square$-graphon}
\def\N{\mathbb N}
\def\R{\mathbb R}
\def\Q{\mathbb Q}
\def\S{\mathbb S}
\def\F{\mathcal F}
\def\K{\mathcal K}
\def\A{\mathcal A}
\def\B{\mathcal B}
\def\C{\mathcal C}
\def\I{\mathcal I}

\begin{document}

\maketitle

\begin{abstract}
We define and study a natural category of graph limits.
The objects are pairs $(\pi,\mu)$, where $\pi$ (the distribution of vertices) is an abstract probability measure on some abstract measurable space $(X,\mathcal A)$ and $\mu$ (the distribution of edges) is an abstract finite measure on the square $(X,\mathcal A)^2$.
Morphisms are random maps between the underlying measurable spaces which preserve the distribution of vertices as well as the distribution of edges.
We also define a convergence notion (inspired by s-convergence) for sequences of graph limits.
We apply tools from category theory to prove the compactness of the space of all graph limits.

\ 


\end{abstract}

%\tableofcontents

\section{Introduction}

It is quite common to  apply category theory in the study of graphs, see e.g. the early survey paper~\cite{Hell}, or the more recent monograph~\cite{HN}.
We devote this paper to showing that category theory may be useful in the world of graph limits.
As it is emphasized in the preface of~\cite{HN}, not only graphs themselves, but also graph homomorphisms should be of central interest in graph theory.
This resembles the general approach from abstract category theory, where morphisms are (at least) at the same level of interest as objects.
This leads to the question whether there is some natural notion of a morphism between two graph limit objects.
We provide such a notion and we use it to prove the compactness of the space of all graph limits (see \Cref{thm:existence}).
While we still have to deal with certain technical difficulties, the main idea of the presented proof is quite straightforward and, as we believe, it nicely demonstrates that category theory may provide a very useful theoretical framework for the study of certain problems arising in the theory of graph limits.

We define graph limits, which we call \sqg s, as pairs $(\pi,\mu)$, where $\pi$ is an abstract probability measure on some abstract measurable space $(X,\A)$ and $\mu$ is an abstract finite measure on the square $(X,\A)^2$.
This is an obvious generalization of graphs: the measure $\pi$ describes the distribution of vertices, while the measure $\mu$ describes the distribution of edges.
Convergence of sequences of \sqg s (and, in particular, convergence of graph sequences) is defined in a very similar way as in~\cite{KLS}.
However, in our abstract setting, it immediately seems to be a good idea to employ the language of category theory.
To accomplish that, we establish morphisms between some pairs of \sqg s.
We define them as random maps (formally represented by Markov kernels) between the underlying measurable spaces which, in a certain sense, preserve the distribution of vertices as well as the distribution of edges.
The values of these random maps are not given deterministically; we only know the probability that a given point is mapped to a given (measurable) set.
This is not a new idea.
Indeed, it was discovered by Lawvere in an unpublished paper~\cite{Lawvere} (and later described by Giry in~\cite{Giry}) that one can define a category with measurable spaces as objects and Markov kernels as morphisms.
For more information on this category, the interested reader may also consult the paper~\cite{DaSi}, where applications to probabilistic programming are presented.
The preservation of measures by Markov kernels was also considered in the above papers, so the only (but crucial) novelty of our definition is the requirement that two measures (the distribution of vertices and the distribution of edges) are preserved at the same time.

The main advantage of our approach is that we do not restrict ourselves to \sqg s on some fixed measurable space.
This allows us to construct \sqg s whose underlying measurable spaces are, for example, products of some given families of other measurable spaces.
Representing such \sqg s, let us say, on the unit interval would only cause unnecessary technical difficulties.
This advantage is clearly demonstrated in the proof of \Cref{thm:existence}, where we prove the compactness of the space of all graph limits.
Namely, the limit of the convergent sequence $(\pi_n,\mu_n)_{n=1}^\infty$ from \Cref{thm:existence} is constructed as a \sqg\ on the infinite product of finite measurable spaces.

Although our definition of convergence, described by convergence of all $k$-shapes in the corresponding Vietoris topologies, is heavily inspired by s-convergence from~\cite{KLS}, let us emphasize one subtle difference.
For every $k\in\N$, we define the $k$-shape of a given \sqg\ as a set of certain weighted graphs on $k$ vertices, where weights are assigned to all edges and to all vertices.
In~\cite{KLS}, the $k$-shape of a given s-graphon can be also interpreted as a set of certain weighted graphs on $k$ vertices, but weights are assigned to edges only.
This corresponds to the fact that we define \sqg s as pairs of measures (one of them representing vertices and the other representing edges), while s-graphons are defined as single measures (representing edges) on the square of the unit interval (where the distribution of vertices is fixed as the Lebesgue measure on the unit interval).
Nevertheless, in \Cref{sec:s-graphons}, we show that this little difference has no real effect on the notion of convergence.

Let us say a few words about the overall strategy of the proof of the compactness of all graph limits (\Cref{thm:existence}).
For a given convergent sequence of \sqg s, we are asked to find its limit.
By the definition of convergence, we know exactly how the $k$-shapes of the limit should look like.
We will use this information to construct two inverse systems of measure spaces (which correspond to the distributions of edges and vertices, respectively) and we find their inverse limits.
These two inverse limits will describe the distribution of vertices and edges, respectively, of a new \sqg.
By the construction, it will follow that this new \sqg\ is, indeed, a limit of our fixed sequence.

In \Cref{sec:s-graphons}, we compare our definition of convergence with s-convergence from~\cite{KLS} and, as a corollary, we reprove the compactness of the space of all s-graphons (which was one of the main results of~\cite{KLS}).

Finally, in \Cref{sec:finalRemarks}, we characterize isomorphisms in the category of all \sqg s (on reasonable measurable spaces) and we start exploring the equivalence relation of `having the same $k$-shapes'.

Some of our preliminary results (including most of the lemmata in \Cref{sec:preliminaries} and \ref{sec:category}) follow readily from well-known facts in category theory.
Nevertheless, we aim to keep the preliminary parts of the paper as elementary as possible, so that readers from other areas can easily follow our arguments.
For this reason, our proofs of the preliminary results rely primarily on measure-theoretic methods, ensuring that the paper remains accessible even to readers with little or no background in category theory.

\section{Preliminaries}
\label{sec:preliminaries}

\subsection{Markov kernels}

A \emph{Markov kernel} from a measurable space $(X,\A)$ to another measurable space $(Y,\B)$ is a map $\kappa\colon X\times\B\to[0,1]$ with the following properties:
\begin{itemize}
\item for every $x\in X$, the map $\kappa(x,\cdot)$ is a probability measure on $(Y,\B)$,
\item for every $B\in\B$, the map $\kappa(\cdot,B)$ is measurable.
\end{itemize}

To simplify the notation, for a bounded measurable function $f\colon Y\to\R$ and $x\in X$, we write shortly $\int_Yf(y)\kappa(x,dy)$ instead of $\int_Yf(y)\,d\kappa(x,\cdot)(y)$.

\medskip

We will often interpret maps between measurable spaces as Markov kernels in the following way.
Let $(X,\A)$, $(Y,\B)$ be measurable spaces. A measurable map $f\colon X\to Y$ is represented by the Markov kernel $\kappa_f$ from $(X,\A)$ to $(Y,\B)$ given by
\begin{equation*}
\kappa_f(x,B)=
\begin{cases}
1,&f(x)\in B,\\
0,&f(x)\notin B,
\end{cases}
\quad x\in X,B\in\B.
\end{equation*}
So, the probability measure $\kappa(x,\cdot)$ is the Dirac measure concentrated at $f(x)$ whenever the singleton set $\{f(x)\}$ is measurable.

\medskip

For later references, we present some auxiliary results here.

Let $L$ be a non-empty finite set.
Let $(X,\A)$ and $(Y_l,\B_l)$, $l\in L$, be measurable spaces.
Let $\kappa_l$ be a Markov kernel from $(X,\A)$ to $(Y_l,\B_l)$, $l\in L$.
Then we define a map $\prod_{l\in L}\kappa_l\colon X\times(\prod_{l\in L}\B_l)\to[0,1]$ (here, $\prod_{l\in L}\B_l$ denotes the corresponding product $\sigma$-algebra) by
\begin{equation}
\label{eq:soucinKap}
\Big(\prod_{l\in L}\kappa_l\Big)(x,R)=\Big(\prod_{l\in L}\kappa_l(x,\cdot)\Big)(R),\quad x\in X,R\in\prod_{l\in L}\B_l,
\end{equation}
where $\prod_{l\in L}\kappa_l(x,\cdot)$ is the corresponding product measure on the measurable space $\prod_{l\in L}\big(Y_l,\B_l\big)$.

\begin{lemma}
\label{lem:measurability}
Let $L$ be a non-empty finite set.
Let $(X,\A)$ and $(Y_l,\B_l)$, $l\in L$, be measurable spaces.
Let $\kappa_l$ be a Markov kernel from $(X,\A)$ to $(Y_l,\B_l)$, $l\in L$.
Then the map $\prod_{l\in L}\kappa_l$ defined by~\eqref{eq:soucinKap} is a Markov kernel from $(X,\A)$ to $\prod_{l\in L}(Y_l,\B_l)$.
\end{lemma}

\begin{proof}
The only nontrivial thing to prove is that, for every fixed set $R\in\prod_{l\in L}\B_l$, the map
\begin{equation*}
x\mapsto\Big(\prod_{l\in L}\kappa_l(x,\cdot)\Big)(R),\quad x\in X,
\end{equation*}
is measurable.
Let $\mathcal R$ be the collection of all sets $R\in\prod_{l\in L}\B_l$ for which this holds true.
It is clear that all measurable rectangles (i.e., sets of the form $\prod_{l\in L}B_l$, where $B_l\in\B_l$, $l\in L$) belong to $\mathcal R$.
Moreover, it is easy to check that the collection $\mathcal R$ is closed under complements and countable disjoint unions.
So, by the $\pi$-$\lambda$ theorem (see~\cite[Theorem~10.1 (iii)]{Kechris}), it holds that $\mathcal R=\prod_{l\in L}\B_l$.
\end{proof}

Let $(X,\A)$, $(Y,\B)$ be measurable spaces.
Let $\kappa$ be a Markov kernel from $(X,\A)$ to $(Y,\B)$.
Then we define a map $\kappa^{\otimes 2}\colon X^2\times\B^2\to[0,1]$ (here, again, $\B^2$ is the corresponding product $\sigma$-algebra) by
\begin{equation}
\label{eq:mocninaKapy}
\kappa^{\otimes 2}((x_1,x_2),\widetilde B)=(\kappa(x_1,\cdot)\times\kappa(x_2,\cdot))(\widetilde B),\quad (x_1,x_2)\in X^2,\widetilde B\in\B^2.
\end{equation}
Similarly, for a measurable map $f\colon X\to Y$, let $f^{\otimes 2}\colon X^2\to Y^2$ be the (measurable) map given by
\[f^{\otimes 2}(x_1,x_2)=(f(x_1),f(x_2)),\quad(x_1,x_2)\in X^2.\]

\begin{example}
Let $(X,\A)$, $(Y,\B)$ be measurable spaces and $f\colon X\to Y$ be a measurable map.
Then $\kappa_f^{\otimes 2}=\kappa_{f^{\otimes 2}}$.
\end{example}

\begin{lemma}
Let $(X,\A)$, $(Y,\B)$ be measurable spaces.
Let $\kappa$ be a Markov kernel from $(X,\A)$ to $(Y,\B)$.
Then the map $\kappa^{\otimes 2}$ is a Markov kernel from $(X,\A)^2$ to $(Y,\B)^2$.
\end{lemma}

\begin{proof}
It is easy to check that, for both $i=1,2$, the map
\begin{equation*}
((x_1,x_2),B)\mapsto\kappa(x_i,B),\quad(x_1,x_2)\in X^2,B\in\B,
\end{equation*}
is a Markov kernel from $(X,\A)^2$ to $(Y,\B)$.
Now it is enough to apply \Cref{lem:measurability}.
\end{proof}

\subsection{Categories}

By a \emph{category}, we mean a structure consisting of
\begin{itemize}
\item a collection of \emph{objects},
\item a collection of \emph{morphisms} from $a$ to $b$, whenever $a$ and $b$ are objects,
\item a partial binary operation $\circ$ on morphisms,
\end{itemize}
which satisfies the following conditions:
\begin{itemize}
\item $g\circ f$ is defined if and only if there are objects $a,b,c$ such that $f$ is a morphism from $a$ to $b$ and $g$ is a morphism from $b$ to $c$; in that case, $g\circ f$ is a morphism from $a$ to $c$,
\item if $f$ is a morphism from $a$ to $b$, $g$ is a morphism from $b$ to $c$ and $h$ is a morphism from $c$ to $d$ then $h\circ(g\circ f)=(h\circ g)\circ f$,
\item for every object $a$ there is a morphism $1_a$ from $a$ to $a$ such that, for every object $b$, every morphism $f$ from $a$ to $b$ and every morphism $g$ from $b$ to $a$, it holds that $f\circ 1_a=f$ and $1_a\circ g=g$.
\end{itemize}
The operation $\circ$ is called \emph{composition} and the morphism $1_a$ is called the \emph{identity morphism} for $a$.

A morphism $f$ from $a$ to $b$ is called an \emph{isomorphism} if there exists a morphism $g$ from $b$ to $a$ such that $g\circ f=1_a$ and $f\circ g=1_b$.

\medskip

The following example will play a crucial role when we introduce our category of graph limits in \Cref{sec:category}.

\begin{example}
\label{ex:meas}
In the unpublished paper~\cite{Lawvere} (and in the later paper~\cite{Giry}), the category of measurable spaces was defined as follows:
\begin{itemize}
\item objects are measurable spaces,
\item morphisms are Markov kernels,
\item the composition of a morphism $\kappa$ from $(X,\A)$ to $(Y,\B)$ and a morphism $\kappa'$ from $(Y,\B)$ to $(Z,\C)$ is given by
\begin{equation}
\label{eq:compositionOfMorphisms}
\kappa'\circ\kappa(x,C)=\int_Y\kappa'(y,C)\kappa(x,dy),\quad x\in X,C\in\C.
\end{equation}
\end{itemize}

Let us also mention that, in~\cite{Tobias}, this category (explicitly described in Section 4 of that paper) is shown to be a special case of a so called Markov category.
\end{example}

\begin{example}
Suppose that $(X,\A)$, $(Y,\B)$, $(Z,\C)$ are measurable spaces.
Let $f\colon X\to Y$ and $f'\colon Y\to Z$ be measurable maps.
Then $\kappa_{f'}\circ\kappa_f=\kappa_{f'\circ f}$, where $\kappa_{f'}\circ\kappa_f$ is defined as in~\eqref{eq:compositionOfMorphisms}.
\end{example}

Later, we will need the following basic fact.

\begin{lemma}
\label{l:otimes}
Let $(X,\A)$, $(Y,\B)$, $(Z,\C)$ be measurable spaces.
Let $\kappa$ a Markov kernel from $(X,\A)$ to $(Y,\B)$ and $\kappa'$ be a Markov kernel from $(Y,\B)$ to $(Z,\C)$.
Then
\begin{equation*}
(\kappa'\circ\kappa)^{\otimes 2}=\kappa'^{\otimes 2}\circ\kappa^{\otimes 2},
\end{equation*}
where $\circ$ is the operation given by~\eqref{eq:compositionOfMorphisms}.
\end{lemma}

\begin{proof}
For every $(x_1,x_2)\in X^2$ and every $C_1,C_2\in\C$, it holds that
\begin{equation*}
\begin{split}
&(\kappa'\circ\kappa)^{\otimes 2}((x_1,x_2),C_1\times C_2)\\
=&\kappa'\circ\kappa(x_1,C_1)\kappa'\circ\kappa(x_2,C_2)\\
=&\int_Y\kappa'(y_1,C_1)\kappa(x_1,dy_1)\int_Y\kappa'(y_2,C_2)\kappa(x_2,dy_2)\\
=&\int_{Y^2}\kappa'^{\otimes 2}((y_1,y_2),C_1\times C_2)\kappa^{\otimes 2}((x_1,x_2),d(y_1,y_2))\\
=&\kappa'^{\otimes 2}\circ\kappa^{\otimes 2}((x_1,x_2),C_1\times C_2).
\end{split}
\end{equation*}
So, for every $(x_1,x_2)\in X^2$, the probability measures $(\kappa'\circ\kappa)^{\otimes 2}((x_1,x_2),\cdot)$ and $\kappa'^{\otimes 2}\circ\kappa^{\otimes 2}((x_1,x_2),\cdot)$ coincide on the collection of all measurable rectangles, and so they must be the same.
The conclusion follows.
\end{proof}

\subsection{Pushforward measures}

Let $(X,\A)$, $(Y,\B)$ be measurable spaces and $\kappa$ be a Markov kernel from $(X,\A)$ to $(Y,\B)$.
Let $\pi$ be a measure on $(X,\A)$.
Then we define the \emph{pushforward} $\kappa_*\pi$ of the measure $\pi$ along the Markov kernel $\kappa$ as the measure on $(Y,\B)$ given by
\begin{equation}
\label{eq:pushforward}
\kappa_*\pi(B)=\int_X\kappa(x,B)\,d\pi(x),\quad B\in\B.
\end{equation}

\begin{example}
Let $(X,\A)$, $(Y,\B)$ be measurable spaces and $f\colon X\to Y$ be a measurable map.
Let $\pi$ be a measure on $(X,\A)$ and $f_*\pi$ be the (classical) pushforward of $\pi$ along $f$.
Then $f_*\pi=(\kappa_f)_*\pi$.
\end{example}

We will need the following basic results.

\begin{proposition}
\label{ex:symmetryPositive}
Let $(X,\A)$ and $(Y,\B)$ be measurable spaces, $\kappa$ be a Markov kernel from $(X,\A)$ to $(Y,\B)$ and $\mu$ be a measure on the measurable space $(X,\A)^2$.
Suppose that the measure $\mu$ is symmetric, that is, 
for every $A_1,A_2\in\A$, it holds that $\mu(A_1\times A_2)=\mu(A_2\times A_1)$.
Then the measure $\kappa^{\otimes2}_*\mu$ on the measurable space $(Y,\B)^2$ is symmetric, as well.
\end{proposition}

\begin{proof}
For every $B_1,B_2\in\B$, we have
\begin{equation*}
\begin{split}
\kappa^{\otimes2}_*\mu(B_1\times B_2)=&\int_{X^2}\kappa(x_1,B_1)\kappa(x_2,B_2)\,d\mu(x_1,x_2)\\
=&\int_{X^2}\kappa(x_1,B_2)\kappa(x_2,B_1)\,d\mu(x_1,x_2)\\
=&\kappa^{\otimes2}_*\mu(B_2\times B_1).
\end{split}
\end{equation*}
\end{proof}

\begin{lemma}
\label{lem:zachovaniMiry}
Let $(X,\A)$, $(Y,\B)$ be measurable spaces and $\kappa$ be a Markov kernel from $(X,\A)$ to $(Y,\B)$.
Let $\pi$ be a measure on $(X,\A)$.
Then $\kappa_*\pi(Y)=\pi(X)$.
\end{lemma}

\begin{proof}
We have
\[
\kappa_*\pi(Y)=\int_X\kappa(x,Y)\,d\pi(x)=\pi(X).
\]
\end{proof}

\begin{lemma}
\label{lem:obrazMiry1}
Let $(X,\A),(Y,\B)$ be measurable spaces.
Let $\pi$ be a measure on $(X,\A)$ and $\kappa$ be a Markov kernel from $(X,\A)$ to $(Y,\B)$.
Then, for every measurable function $f\colon Y\to[0,\infty)$, it holds that
\begin{equation}
\label{eq:obrazFunkce}
\int_Yf(y)\,d\kappa_*\pi(y)=\int_X\int_Yf(y)\kappa(x,dy)\,d\pi(x).
\end{equation}
\end{lemma}

\begin{proof}
If $f$ is a characteristic function of a measurable set, then~\eqref{eq:obrazFunkce} is just a reformulation of~\eqref{eq:pushforward}.
If $f$ is a simple function, we apply the linearity of integration.
The general case follows by monotone convergence theorem.
\end{proof}

\begin{corollary}
\label{cor:slozeni}
Let $(X,\A),(Y,\B),(Z,\C)$ be measurable spaces.
Let $\pi$ be a measure on $(X,\A)$.
Let $\kappa$ be a Markov kernel from $(X,\A)$ to $(Y,\B)$ and $\kappa'$ be a Markov kernel from $(Y,\B)$ to $(Z,\C)$.
Then
\begin{equation*}
(\kappa'\circ\kappa)_*\pi=\kappa'_*(\kappa_*\pi),
\end{equation*}
where $\kappa'\circ\kappa$ is the Markov kernel from $(X,\A)$ to $(Z,\C)$ given by~\eqref{eq:compositionOfMorphisms}.
\end{corollary}

\begin{proof}
For every $C\in\C$, it holds that
\begin{equation*}
\begin{split}
(\kappa'\circ\kappa)_*\pi(C)=&\int_X\kappa'\circ\kappa(x,C)\,d\pi(x)\\
=&\int_X\int_Y\kappa'(y,C)\kappa(x,dy)\,d\pi(x)\\
\stackrel{\Cref{lem:obrazMiry1}}{=}&\int_Y\kappa'(y,C)\,d\kappa_*\pi(y)\\
=&\kappa'_*(\kappa_*\pi)(C).
\end{split}
\end{equation*}
\end{proof}

\subsection{Inverse systems of measures and Kolmogorov Extension Theorem}
\label{sec:inverse}

Inverse systems of measure spaces, as well as their inverse limits, were introduced in~\cite[Chapter~5]{Bochner} as a natural generalization of product measures.
Although inverse limits of inverse systems of measure spaces do not always exist (see, e.g., \cite{AJ}), there are many instances where the existence is guaranteed.
The aim of this subsection is to recall one such case in which the inverse system is of a very special form.
In fact, it is just a special case of a result widely known as the Kolmogorov Extension Theorem (see e.g.~\cite[Theorem~2.4.3]{Tao}).

To keep things simple, we refrain from repeating the general definitions of inverse systems of measure spaces and their inverse limits (the interested reader can find them in~\cite{Bochner}, or in~\cite{Choksi}).
Instead, we consider only a simplified case, given by the following additional assumptions:
First, we assume that the inverse system is indexed by a countable set.
Second, we assume that all the measure spaces from the inverse system have finite underlying sets.
Third, we assume that all the bounding maps of the inverse system are projections (which also requires to assume a special form of the underlying sets).
The second condition (which could be relaxed even more), together with the third one, allow us to directly apply the Kolmogorov Extension Theorem.
With all this in mind, we define only the following special cases of inverse systems of measure spaces and their inverse limits.

\medskip

Let $D$ be a countable index set.
For every $d\in D$, let $F_d$ be a non-empty finite set and $2^{F_d}$ be the discrete $\sigma$-algebra on $F_d$.
Let $\I$ be the set of all non-empty finite subsets of $D$.
For every $I\in\I$, let $\rho^I$ be a finite measure on $\prod_{d\in I}(F_d,2^{F_d})$.
Then, for the purposes of this paper, we say that the collection $(\rho^I)_{I\in\I}$ is a \emph{simple inverse system} of measures if
\begin{equation}
\label{eq:IS}
\rho^I=(P_{I,J})_*\rho^J,\quad I,J\in\I,I\subseteq J,
\end{equation}
where $P_{I,J}$ is the canonical projection from $\prod_{d\in J}F_d$ to $\prod_{d\in I}F_d$.
Moreover, we say that a finite measure $\rho$ on $\prod_{d\in D}(F_d,2^{F_d})$ is the \emph{inverse limit} of the simple inverse system $(\rho^I)_{I\in\I}$ if
\begin{equation}
\label{eq:IL}
\rho^I=(P_I)_*\rho,\quad I\in\I,
\end{equation}
where $P_I$ is the canonical projection from $\prod_{d\in D}F_d$ to $\prod_{d\in I}F_d$.

\begin{lemma}[Kolmogorov Extension Theorem]
\label{lem:existenceInverseLimit}
Every simple inverse system of measures has a unique inverse limit.
\end{lemma}

\section{The category of \sqg s}
\label{sec:category}

The key notion of this paper is the following generalization of a graph.

\begin{definition}
Let $(X,\A)$ be a measurable space.
A \emph{\sqg} on $(X,\A)$ is a pair $(\pi,\mu)$ where $\pi$ is a probability measure on $(X,\A)$ and $\mu$ is a finite measure on $(X,\A)^2$.
\end{definition}

To define the categorical structure on the class of all \sqg s, we first introduce the notion of a morphism.

\begin{definition}
\label{def:morphism}
Let $(\pi_X,\mu_X)$ be a \sqg\ on a measurable space $(X,\A)$ and $(\pi_Y,\mu_Y)$ be a \sqg\ on a measurable space $(Y,\B)$.
A \emph{morphism} from $(\pi_X,\mu_X)$ to $(\pi_Y,\mu_Y)$ is a Markov kernel $\kappa$ from $(X,\A)$ to $(Y,\B)$ such that $\pi_Y=\kappa_*\pi_X$ and $\mu_Y=\kappa^{\otimes 2}_*\mu_X$.
\end{definition}

Because of the importance of the definition above, we explicitly write down the formula for the measure $\kappa^{\otimes 2}_*\mu_X$ (which is obtained by putting~\eqref{eq:mocninaKapy} and~\eqref{eq:pushforward} together):
\begin{equation*}
\kappa^{\otimes 2}_*\mu_X(\widetilde B)=\int_{X^2}(\kappa(x_1,\cdot)\times\kappa(x_2,\cdot))(\widetilde B)\,d\mu_X(x_1,x_2),\quad\widetilde B\in\B^2.
\end{equation*}
It is a classical result (easily obtainable from the $\pi$-$\lambda$ theorem) that, if two finite measures on $(Y,\B)^2$ coincide on the collection of all measurable rectangles, then they are the same.
So, to verify that $\mu_Y=\kappa^{\otimes 2}_*\mu_X$, it is enough to check that
\begin{equation*}
\mu_Y(B_1\times B_2)=\int_{X^2}\kappa(x_1,B_1)\kappa(x_2,B_2)\,d\mu_X(x_1,x_2),\quad B_1,B_2\in\B.
\end{equation*}

\begin{example}
\label{ex:deterministicMorphism}
Let $(\pi_X,\mu_X)$ be a \sqg\ on a measurable space $(X,\A)$ and $(\pi_Y,\mu_Y)$ be a \sqg\ on a measurable space $(Y,\B)$.
Let $f\colon X\to Y$ be a measurable map.
Then $\kappa_f$ is a morphism from $(\pi_X,\mu_X)$ to $(\pi_Y,\mu_Y)$ if and only if $\pi_Y=f_*\pi_X$ and $\mu_Y=f^{\otimes 2}_*\mu_X$.
\end{example}

Next, we introduce the composition operation.

\begin{definition}
\label{def:composition}
Suppose that $(\pi_X,\mu_X)$, $(\pi_Y,\mu_Y)$, $(\pi_Z,\mu_Z)$ are \sqg s on measurable spaces $(X,\A)$, $(Y,\B)$, $(Z,\C)$, respectively.
Let $\kappa$ be a morphism from $(\pi_X,\mu_X)$ to $(\pi_Y,\mu_Y)$ and $\kappa'$ be a morphism from $(\pi_Y,\mu_Y)$ to $(\pi_Z,\mu_Z)$.
Then we define the morphism $\kappa'\circ\kappa$ from $(\pi_X,\mu_X)$ to $(\pi_Z,\mu_Z)$ by~\eqref{eq:compositionOfMorphisms}.
\end{definition}

To verify the corectness of \Cref{def:composition}, it is necessary to check that the Markov kernel $\kappa'\circ\kappa$ from $(X,\A)$ to $(Z,\C)$ defined by~\eqref{eq:compositionOfMorphisms} satisfies $(\kappa'\circ\kappa)_*\pi_X=\pi_Z$ and $(\kappa'\circ\kappa)^{\otimes 2}_*\mu_X=\mu_Z$.
The former equation follows immediately from \Cref{cor:slozeni}.
The latter equation follows by an easy combination of \Cref{l:otimes} and \Cref{cor:slozeni}.

\medskip

It is straightforward to verify that morphisms from \Cref{def:morphism}, together with the composition operation from \Cref{def:composition}, define a categorical structure on the class of all \sqg s. Indeed, the associativity of the composition operation is granted for free by the associativity in the category of measurable spaces from \Cref{ex:meas} (it can also be easily verified using Fubini's theorem).
Now assume that $(\pi_X,\mu_X)$ is a \sqg\ on a measurable space $(X,\A)$, and let $i_X$ be the identity map on $X$.
Then the Markov kernel $\kappa_{i_X}$ from $(X,\A)$ to $(X,\A)$ is the identity morphism for $(\pi_X,\mu_X)$.

\medskip

We conclude this section by a few examples.

\begin{example}
Let $G=(V,E)$ be a finite graph with a non-empty vertex set $V$.
We may allow any combination of the following: loops, multiple edges, weighted vertices and/or edges, directed edges.
Then $G$ can be naturally understood as a \sqg\ $(\pi_G,\mu_G)$ on the measurable space $(V,2^V)$, where $2^V$ is the discrete $\sigma$-algebra on the vertex set $V$.
Indeed, let $\pi_G$ be the normalized counting measure on $(V,2^V)$ (with an obvious modification in the case of weighted vertices) and let $\mu_G$ be the measure on $(V,2^V)^2$ defined such that, for every $v_1,v_2\in V$, $\mu_G(\{(v_1,v_2)\})$ equals the number of edges from $v_1$ to $v_2$ (with an obvious modification in the case of weighted edges).

Sometimes, it is useful to additionally use some kind of normalization to the measure $\mu_G$.
\end{example}

\begin{example}
A \emph{graphon (in the general form)} is defined as a pair $(\mathcal X,W)$ where $\mathcal X=(X,\A,\pi_X)$ is a probability space and $W\colon X^2\to [0,1]$ is a symmetric function which is measurable with respect to the completion of the $\sigma$-algebra $\A^2$ (see \cite[p. 217]{Lovasz}).

Every graphon $(\mathcal X,W)$, where $\mathcal X=(X,\A,\pi_X)$, can be naturally understood as a \sqg\ $(\pi_X,\mu_W)$ on the measurable space $(X,\A)$.
Indeed, just define
\begin{equation*}
\mu_W(\widetilde A)=\int_{\widetilde A}W(x_1,x_2)\,d\pi_X^2(x_1,x_2),\quad\widetilde A\in\A^2.
\end{equation*}
\end{example}

\begin{example}
In ~\cite[Definition~4.6]{KLS}, s-graphons are defined as symmetric Borel probability measures on the square of the unit interval (which is equipped with the Borel $\sigma$-algebra and with the Lebesgue measure).
Then every s-graphon can be naturally viewed as a special case of a \sqg.
\end{example}

\begin{example}
\label{ex:newGraphons}
Let $(X,\A)$, $(Y,\B)$ be measurable spaces.
Suppose that $(\pi,\mu)$ is a \sqg\ on $(X,\A)$.
Let $\kappa$ be a Markov kernel from $(X,\A)$ to $(Y,\B)$.
Then $(\kappa_*\pi,\kappa^{\otimes 2}_*\mu)$ is a \sqg\ on $(Y,\B)$ and $\kappa$ is a morphism from $(\pi,\mu)$ to $(\kappa_*\pi,\kappa^{\otimes 2}_*\mu)$.
\end{example}

\begin{example}
\label{ex:trivialMorphisms}
Let $\Omega_j$ be a non-empty finite set and $2^{\Omega_j}$ be the discrete $\sigma$-algebra on $\Omega_j$, $j=1,2$.
Suppose that $f\colon\Omega_2\to\Omega_1$ is a surjective map.
Let $(\rho_1,\nu_1)$ be a \sqg\ on $(\Omega_1,2^{\Omega_1})$.
Let $\kappa$ be the Markov kernel from $(\Omega_1,2^{\Omega_1})$ to $(\Omega_2,2^{\Omega_2})$ given by
\[
\kappa(\omega_1,F)=\frac{|F\cap f^{-1}(\{\omega_1\})|}{|f^{-1}(\{\omega_1\})|},\quad\omega_1\in\Omega_1,F\subseteq\Omega_2.
\]
Let $(\rho_2,\nu_2)$ be the \sqg\ on $(\Omega_2,2^{\Omega_2})$ defined by
\[
\rho_2=\kappa_*\rho_1\quad\text{and}\quad\nu_2=\kappa^{\otimes2}_*\nu_1
\]
(so that $\kappa$ is a morphism from $(\rho_1,\nu_1)$ to $(\rho_2,\nu_2)$).
Then the Markov kernel $\kappa_f$ from $(\Omega_2,2^{\Omega_2})$ to $(\Omega_1,2^{\Omega_1})$ is a morphism from $(\rho_2,\nu_2)$ to $(\rho_1,\nu_1)$.
\end{example}

\begin{example}
\label{ex:nonatomic}
Let $(\pi_X,\mu_X)$ be a \sqg\ on a measurable space $(X,\A)$.
Suppose that all singleton subsets of $X$ are measurable and put
\[
A=\big\{x\in X:\pi_X(\{x\})>0\big\}.
\]
Let $(Y,\B)$ be the measurable space obtained from $(X,\A)$ by replacing each $x\in A$ by a copy $I_x$ of the interval $[0,\pi_X(\{x\})]$ with its Borel $\sigma$-algebra $\B_x$.
More precisely, we put
\[
Y=(X\setminus A)\cup\bigcup_{x\in A}I_x
\]
and
\[
\B=\big\{B\subseteq Y:B\setminus A\in\A\text{ and }B\cap I_x\in\B_x,x\in A\big\}.
\]
For every $x\in A$, let $\lambda_x$ be the Lebesgue measure on $(I_x,\B_x)$.
Let $\kappa$ be the Markov kernel from $(X,\A)$ to $(Y,\B)$ given by
\begin{multline*}
\kappa(x,B)=
\begin{cases}
1,&x\in X\setminus A\text{ and }x\in B,\\
0,&x\in X\setminus A\text{ and }x\notin B,\\
\frac{\lambda_x(B\cap I_x)}{\pi_X(\{x\})},&x\in A,
\end{cases}
\\
B\in\B.
\end{multline*}
Let $(\pi_Y,\mu_Y)$ be the \sqg\ on $(Y,\B)$ defined by
\[
\pi_Y=\kappa_*\pi_X\quad\text{and}\quad\mu_Y=\kappa^{\otimes2}_*\mu_X
\]
(so that $\kappa$ is a morphism from $(\pi_X,\mu_X)$ to $(\pi_Y,\mu_Y)$).
Let $f\colon Y\to X$ be the (measurable) map defined by
\[
f(y)=
\begin{cases}
y,&y\in X\setminus A,\\
x,&y\in I_x,x\in A.
\end{cases}
\]
Then the Markov kernel $\kappa_f$ from $(Y,\B)$ to $(X,\A)$ is a morphism from $(\pi_Y,\mu_Y)$ to $(\pi_X,\mu_X)$.

Obviously, if the measurable space $(X,\A)$ is standard Borel (see, e.g., \cite[Section 12.B]{Kechris}) then $(Y,\B)$ is stadard Borel as well.
\end{example}

\section{Convergence}

In this section, we introduce so called $k$-shapes and the notion of convergence for sequences of \sqg s.
Both these notions are heavily inspired by definitions of $k$-shapes and s-convergence from~\cite{KLS}.

\medskip

For every $k\in\N$, we define a measurable space $\F_k$ by
\[\F_k=\big([k],2^{[k]}\big),\]
where $2^{[k]}$ is the discrete $\sigma$-algebra on the set $[k]=\{1,\ldots,k\}$.
For every \sqg\ $(\rho,\nu)$ on $\F_k$, we can define $e_k(\rho,\nu)\in\R^{[k]\cup[k]^2}$ by
\[e_k(\rho,\nu)(i)=\rho(\{i\}),\quad i\in[k],\]
and
\[e_k(\rho,\nu)(i,j)=\nu(\{(i,j)\}),\quad(i,j)\in[k]^2.\]
Then $e_k$ is an embedding of the set of all \sqg s on $\F_k$ into $\R^{[k]\cup[k]^2}$.
We equip the space of all \sqg s on $\F_k$ with the topology which makes the embedding $e_k$ a homeomorphism of the space of all \sqg s on $\F_k$ onto its image (which is endowed with the topology inherited from $\R^{[k]\cup[k]^2}$).
In particular, a sequence $(\rho_n,\nu_n)_{n=1}^\infty$ of \sqg s on $\F_k$ is convergent to a \sqg\ $(\rho,\nu)$ on $\F_k$ if and only if
\[\lim_{n\to\infty}\rho_n(\{i\})=\rho(\{i\}),\quad i\in[k],\]
and
\[\lim_{n\to\infty}\nu_n(\{(i,j)\})=\nu(\{(i,j)\}),\quad(i,j)\in[k]^2.\]

\medskip

Let $\preceq$ be the preorder on the class of all \sqg s given by
\begin{equation*}
(\pi_Y,\mu_Y)\preceq(\pi_X,\mu_X)\quad\Leftrightarrow\quad\text{ there exists a morphism from }(\pi_X,\mu_X)\text{ to }(\pi_Y,\mu_Y).
\end{equation*}
For a given \sqg\ $(\pi,\mu)$, let $(\pi,\mu)^\downarrow$ be the $\preceq$-downward closure of $(\pi,\mu)$.
That is, $(\pi,\mu)^\downarrow$ is the class of all \sqg s $(\rho,\nu)\preceq(\pi,\mu)$.
For a \sqg\ $(\pi,\mu)$ and $k\in\N$, we also define
\[(\pi,\mu)^\downarrow_k=\big\{(\rho,\nu)\in(\pi,\mu)^\downarrow:(\rho,\nu)\text{ is a \sqg\ on $\F_k$}\big\}.\]

\begin{definition}
\label{def:k-shape}
Suppose that $(\pi,\mu)$ is a \sqg.
Then, for every $k\in\N$, we define the \emph{$k$-shape} $\S_k(\pi,\mu)$ of $(\pi,\mu)$ as the topological closure of $(\pi,\mu)^\downarrow_k$ in the space of all \sqg s on $\F_k$.
\end{definition}

\begin{lemma}
\label{lem:compactness}
For every \sqg\ $(\pi,\mu)$ and every $k\in\N$, the $k$-shape $\S_k(\pi,\mu)$ is a nonempty compact set.
\end{lemma}

\begin{proof}
We fix $k\in\N$ and a \sqg\ $(\pi,\mu)$ on a measurable space $(X,\A)$.
For every Markov kernel $\kappa$ from $(X,\A)$ to $\F_k$, the \sqg\ $(\kappa_*\pi,\kappa^{\otimes 2}_*\mu)$ belongs to $(\pi,\mu)^\downarrow_k$ (by \Cref{ex:newGraphons}).
In particular, the set $(\pi,\mu)^\downarrow_k$ is non-empty.
Consequently, the same is true for the $k$-shape $\S_k(\pi,\mu)$.

It remains to show that $\S_k(\pi,\mu)$ is a compact set.
That is, we must show that the closure of $e_k\big((\pi,\mu)^\downarrow_k\big)$ in the image of the embedding $e_k$ is a compact set.
We note that the image of $e_k$ is a closed subset of $\R^{[k]\cup[k]^2}$.
Indeed, it contains exactly those $\alpha\in\R^{[k]\cup[k]^2}$ which have all entries non-negative and which satisfy $\sum_{i\in [k]}\alpha(i)=1$.
So the closure of $e_k\big((\pi,\mu)^\downarrow_k\big)$ in the image of $e_k$ is the same as its closure in $\R^{[k]\cup[k]^2}$.
So, to prove the compactness, it is enough to verify that $e_k\big((\pi,\mu)^\downarrow_k\big)$ is a bounded subset of $\R^{[k]\cup[k]^2}$.

For every $(\rho,\nu)\in(\pi,\mu)^\downarrow_k$, there is a morphism from $(\pi,\mu)$ to $(\rho,\nu)$, and so $\nu([k]^2)=\mu(X^2)$ (by \Cref{lem:zachovaniMiry}).
This, together with the fact that $\rho$ is a probability measure on $\F_k$, gives us that
\begin{equation*}
\begin{split}
&\sum_{i\in[k]}e_k(\rho,\nu)(i)+\sum_{(i,j)\in[k]^2}e_k(\rho,\nu)(i,j)\\
=&\sum_{i\in[k]}\rho(\{i\})+\sum_{(i,j)\in[k]^2}\nu(\{(i,j)\})\\
=&1+\mu(X^2).
\end{split}
\end{equation*}
So the set $e_k\big((\pi,\mu)^\downarrow_k\big)$ is bounded, which completes the proof.
\end{proof}

For every $k\in\N$, let $\K_k$ be the space of all non-empty compact subsets of the space of all \sqg s on the measurable space $\F_k$.
We equip the space $\K_k$ with the Vietoris topology (see, e.g., \cite[Section 4.F]{Kechris}).
By \Cref{lem:compactness}, each $k$-shape is an element of $\K_k$.

\begin{definition}
\label{def:convergence}
We say that a sequence $(\pi_n,\mu_n)_{n=1}^\infty$ of \sqg s is \emph{convergent} if, for every $k\in\N$, the sequence $\big(\S_k(\pi_n,\mu_n)\big)_{n=1}^\infty$ of the corresponding $k$-shapes is convergent in $\K_k$.

If, moreover, $(\pi,\mu)$ is another \sqg\ then we say that $(\pi,\mu)$ is a \emph{limit} of the sequence $(\pi_n,\mu_n)_{n=1}^\infty$ if, for every $k\in\N$, the $k$-shape $\S_k(\pi,\mu)$ of $(\pi,\mu)$ is the limit of the sequence $\big(\S_k(\pi_n,\mu_n)\big)_{n=1}^\infty$ in $\K_k$.
\end{definition}

We conclude this section by an easy example.

\begin{example}
\label{ex:symmetry}
Suppose that $(\pi,\mu)$ is a \sqg\ on a measurable space $(X,\A)$.
Suppose also that the measure $\mu$ is not symmetric.
Then there is $(\rho,\nu)\in(\pi,\mu)^\downarrow_3\subseteq\S_3(\pi,\mu)$ such that $\nu$ is not symmetric.

In order to see it, fix $A_1,A_2\in\A$ such that $\mu(A_1\times A_2)\neq\mu(A_2\times A_1)$.
If the sets $A_1,A_2$ are not disjoint, then it is easy to see that either
\[
\mu((A_1\setminus A_2)\times A_2)\neq\mu(A_2\times(A_1\setminus A_2)),
\]
or
\[
\mu((A_1\cap A_2)\times(A_2\setminus A_1))\neq\mu((A_2\setminus A_1)\times(A_1\cap A_2)).
\]
So, without loss of generality, we may assume that $A_1\cap A_2=\emptyset$.

Let $f\colon X\to[3]$ be the (measurable) map given by
\begin{equation*}
f(x)=
\begin{cases}
1,&x\in A_1,\\
2,&x\in A_2,\\
3,&x\in X\setminus(A_1\cup A_2).
\end{cases}
\end{equation*}
We put
\[
\rho=f_*\pi\quad\text{and}\quad\nu=f^{\otimes2}_*\mu.
\]
Then $(\rho,\nu)\in(\pi,\mu)^\downarrow_3$ (as witnessed by the morphism $\kappa_f$) and
\begin{equation*}
\begin{split}
\nu(\{1\}\times\{2\})=&f^{\otimes2}_*\mu(\{1\}\times\{2\})\\
=&\mu(A_1\times A_2)\\
\neq&\mu(A_2\times A_1)\\
=&f^{\otimes2}_*\mu(\{2\}\times\{1\})\\
=&\nu(\{2\}\times\{1\}).
\end{split}
\end{equation*}
\end{example}

\subsection{Existence of limits}

This subsection is devoted to the proof of the following theorem.

\begin{theorem}
\label{thm:existence}
Every convergent sequence $(\pi_n,\mu_n)_{n=1}^\infty$ of \sqg s has a limit.
\end{theorem}

We start with some auxiliary lemmata.

\begin{lemma}
\label{lem:projekceJakoMorfismy}
Let $L_1\subseteq L_2$ be non-empty finite sets.
Let $(\pi,\mu)$ be a \sqg\ on a measurable space $(X,\A)$ and, for every $l\in L_2$, let $(\rho_l,\nu_l)$ be a \sqg\ on a measurable space $(Y_l,\B_l)$.
Suppose that $\kappa_l$ is a morphism from $(\pi,\mu)$ to $(\rho_l,\nu_l)$, $l\in L_2$.
We put
\[\rho^{L_i}=\Big(\prod_{l\in L_i}\kappa_l\Big)_*\pi\quad\text{and}\quad\nu^{L_i}=\Big(\prod_{l\in L_i}\kappa_l\Big)^{\otimes2}_*\mu,\quad i=1,2.\]
Let $P_{L_1L_2}$ be the canonical projection from $\prod_{l\in L_2}Y_l$ to $\prod_{l\in L_1}Y_l$.
Then the Markov kernel $\kappa_{P_{L_1L_2}}$ is a morphism from $(\rho^{L_2},\nu^{L_2})$ to $(\rho^{L_1},\nu^{L_1})$.
\end{lemma}

\begin{proof}
We must show that
\begin{equation}
\label{eq:jedna}
\rho^{L_1}=(P_{L_1L_2})_*\rho^{L_2}
\end{equation}
and
\begin{equation}
\label{eq:dva}
\nu^{L_1}=(P_{L_1L_2})^{\otimes 2}_*\nu^{L_2}.
\end{equation}
For every $R\in\prod_{l\in L_1}\B_l$ and every every $x\in X$, it holds that
\begin{equation}
\label{eq:projekce}
\begin{split}
\Big(\prod_{l\in L_1}\kappa_l\Big)(x,R)\stackrel{\eqref{eq:soucinKap}}{=}&\Big(\prod_{l\in L_1}\kappa_l(x,\cdot)\Big)(R)\\
=&\Big(\prod_{l\in L_2}\kappa_l(x,\cdot)\Big)\Big(P_{L_1L_2}^{-1}(R)\Big)\\
\stackrel{\eqref{eq:soucinKap}}{=}&\Big(\prod_{l\in L_2}\kappa_l\Big)\Big(x,P_{L_1L_2}^{-1}(R)\Big).
\end{split}
\end{equation}
So, for every $R\in\prod_{l\in L_1}\B_l$, we have
\begin{equation*}
\begin{split}
\rho^{L_1}(R)=&\Big(\prod_{l\in L_1}\kappa_l\Big)_*\pi(R)\\
=&\int_X\Big(\prod_{l\in L_1}\kappa_l\Big)(x,R)\,d\pi(x)\\
\stackrel{\eqref{eq:projekce}}{=}&\int_X\Big(\prod_{l\in L_2}\kappa_l\Big)(x,P_{L_1L_2}^{-1}(R))\,d\pi(x)\\
=&\Big(\prod_{l\in L_2}\kappa_l\Big)_*\pi\big(P_{L_1L_2}^{-1}(R)\big)\\
=&\rho^{L_2}\big(P_{L_1L_2}^{-1}(R)\big),
\end{split}
\end{equation*}
which verifies~\eqref{eq:jedna}.

Similarly, for every $R_1,R_2\in\prod_{l\in L_1}\B_l$ and every $(x_1,x_2)\in X^2$, it holds that
\begin{equation}
\label{eq:projekceB}
\begin{split}
&\Big(\prod_{l\in L_1}\kappa_l\Big)^{\otimes 2}\big((x_1,x_2),R_1\times R_2\big)\\
\stackrel{\eqref{eq:mocninaKapy}}{=}&\Big(\prod_{l\in L_1}\kappa_l\Big)(x_1,R_1)\Big(\prod_{l\in L_1}\kappa_l\Big)(x_2,R_2)\\
\stackrel{\eqref{eq:projekce}}{=}&\Big(\prod_{l\in L_2}\kappa_l\Big)\big(x_1,P_{L_1L_2}^{-1}(R_1)\big)\Big(\prod_{l\in L_2}\kappa_l\Big)\big(x_2,P_{L_1L_2}^{-1}(R_2)\big)\\
\stackrel{\eqref{eq:mocninaKapy}}{=}&\Big(\prod_{l\in L_2}\kappa_l\Big)^{\otimes 2}\Big((x_1,x_2),\big((P_{L_1L_2}^{\otimes2})^{-1}(R_1\times R_2)\big)\Big).
\end{split}
\end{equation}
So, for every $R_1,R_2\in\prod_{l\in L_1}\B_l$, we have
\begin{equation*}
\begin{split}
\nu^{L_1}(R_1\times R_2)=&\Big(\prod_{l\in L_1}\kappa_l\Big)^{\otimes 2}_*\mu(R_1\times R_2)\\
=&\int_{X^2}\Big(\prod_{l\in L_1}\kappa_l\Big)^{\otimes 2}\big((x_1,x_2),R_1\times R_2\big)\,d\mu(x_1,x_2)\\
\stackrel{\eqref{eq:projekceB}}{=}&\int_{X^2}\Big(\prod_{l\in L_2}\kappa_l\Big)^{\otimes 2}\Big((x_1,x_2),\big((P_{L_1L_2}^{\otimes2})^{-1}(R_1\times R_2)\big)\Big)\,d\mu(x_1,x_2)\\
=&\Big(\prod_{l\in L_2}\kappa_l\Big)^{\otimes 2}_*\mu\Big((P_{L_1L_2}^{\otimes2})^{-1}\big(R_1\times R_2\big)\Big)\\
=&\nu^{L_2}\Big((P_{L_1L_2}^{\otimes2})^{-1}\big(R_1\times R_2\big)\Big),
\end{split}
\end{equation*}
which verifies~\eqref{eq:dva}.
\end{proof}

\begin{lemma}
\label{lem:malaZmenaMorphismu}
Let $(\pi,\mu)$ be a \sqg\ on a measurable space $(X,\A)$ and let $k\in\N$ be fixed.
Suppose that $(\rho_i,\nu_i)$ is a \sqg\ on $\F_k$ and that $\kappa_i$ is a morphism from $(\pi,\mu)$ to $(\rho_i,\nu_i)$, $i=1,2$.
Let $P^j\colon X^2\to X$ be the projection on the $j$th coordinate, $j=1,2$.
Let $\gamma$ be the measure on $(X,\A)$ given by
\begin{equation}
\label{triMiry}
\gamma=\pi+P^1_*\mu+P^2_*\mu.
\end{equation}
Let $\eta>0$ be such that
\begin{equation}
\label{smallError}
\big\|\kappa_1(\cdot,\{m\})-\kappa_2(\cdot,\{m\})\big\|_{L^1(\gamma)}<\eta,\quad m\in[k].
\end{equation}
Then it holds that
\begin{equation*}
|\rho_1(\{m\})-\rho_2(\{m\})|<\eta,\quad m\in[k],
\end{equation*}
and
\begin{equation*}
|\nu_1(\{(m_1,m_2)\})-\nu_2(\{(m_1,m_2)\})|<2\eta,\quad (m_1,m_2)\in[k]^2.
\end{equation*}
\end{lemma}

\begin{proof}
For every $m\in[k]$, it holds that
\begin{equation*}
\begin{split}
|\rho_1(\{m\})-\rho_2(\{m\})|=&|(\kappa_1)_*\pi(\{m\})-(\kappa_2)_*\pi(\{m\})|\\
=&\Big|\int_X\kappa_1(x,\{m\})\,d\pi(x)-\int_X\kappa_2(x,\{m\})\,d\pi(x)\Big|\\
\le&\big\|\kappa_1(\cdot,\{m\})-\kappa_2(\cdot,\{m\})\big\|_{L^1(\pi)}\\
\stackrel{\eqref{triMiry},\eqref{smallError}}{<}&\eta.
\end{split}
\end{equation*}
Similarly, for every $(m_1,m_2)\in[k]^2$, it holds that
\begin{equation*}
\begin{split}
&|\nu_1(\{(m_1,m_2)\})-\nu_2(\{(m_1,m_2)\})|\\
=&|(\kappa_1)^{\otimes 2}_*\mu(\{(m_1,m_2)\})-(\kappa_2)^{\otimes 2}_*\mu(\{(m_1,m_2)\})|\\
=&\Big|\int_{X^2}\kappa_1(x_1,\{m_1\})\kappa_1(x_2,\{m_2\})\,d\mu(x_1,x_2)\\
&-\int_{X^2}\kappa_2(x_1,\{m_1\})\kappa_2(x_2,\{m_2\})\,d\mu(x_1,x_2)\Big|\\
\le&\int_{X^2}|\kappa_1(x_1,\{m_1\})\kappa_1(x_2,\{m_2\})-\kappa_2(x_1,\{m_1\})\kappa_2(x_2,\{m_2\})|\,d\mu(x_1,x_2)\\
\le&\int_{X^2}\kappa_1(x_1,\{m_1\})\cdot|\kappa_1(x_2,\{m_2\})-\kappa_2(x_2,\{m_2\})|\,d\mu(x_1,x_2)\\
&+\int_{X^2}\kappa_2(x_2,\{m_2\})\cdot|\kappa_1(x_1,\{m_1\})-\kappa_2(x_1,\{m_1\})|\,d\mu(x_1,x_2)\\
\le&\big\|\kappa_1(\cdot,\{m_2\})-\kappa_2(\cdot,\{m_2\})\big\|_{L^1(P^2_*\mu)}+\big\|\kappa_1(\cdot,\{m_1\})-\kappa_2(\cdot,\{m_1\})\big\|_{L^1(P^1_*\mu)}\\
\stackrel{\eqref{triMiry},\eqref{smallError}}{<}&2\eta.
\end{split}
\end{equation*}

\end{proof}

\begin{lemma}
\label{lem:malaZmenaGraphonu}
Let $(\rho,\nu)$ and $(\rho',\nu')$ be \sqg s, both on the same measurable space $(\Omega,2^\Omega)$, where the set $\Omega$ is non-empty and finite and $2^\Omega$ is the discrete $\sigma$-algebra on $\Omega$.
Let $\kappa$ be a Markov kernel from $(\Omega,2^\Omega)$ to a measurable space $(X,\A)$.
We put
\[
m=\max_{\omega\in\Omega}|\rho(\{\omega\})-\rho'(\{\omega\})|
\]
and
\[
M=\max_{(\omega_1,\omega_2)\in\Omega^2}|\nu(\{(\omega_1,\omega_2)\})-\nu'(\{(\omega_1,\omega_2)\})|.
\]
Then it holds that
\[
|\kappa_*\rho(A)-\kappa_*\rho'(A)|\le m|\Omega|,\quad A\in\A,
\]
and
\[
|\kappa^{\otimes 2}_*\nu(\widetilde A)-\kappa^{\otimes 2}_*\nu'(\widetilde A)|\le M|\Omega|^2,\quad\widetilde A\in\A^2.
\]
\end{lemma}

\begin{proof}
For every $A\in\A$, we have
\begin{equation*}
\begin{split}
|\kappa_*\rho(A)-\kappa_*\rho'(A)|=&\Big|\int_\Omega\kappa(\omega,A)\,d\rho(\omega)-\int_\Omega\kappa(\omega,A)\,d\rho'(\omega)\Big|\\
\le&\sum_{\omega\in\Omega}|\rho(\{\omega\})-\rho'(\{\omega\})|\kappa(\omega,A)\\
\le&m|\Omega|.
\end{split}
\end{equation*}
Similarly, for every $\widetilde A\in\A^2$, we have
\begin{equation*}
\begin{split}
&|\kappa^{\otimes 2}_*\nu(\widetilde A)-\kappa^{\otimes 2}_*\nu'(\widetilde A)|\\
=&\Big|\int_{\Omega^2}\kappa^{\otimes 2}((\omega_1,\omega_2),\widetilde A)\,d\nu(\omega_1,\omega_2)-\int_{\Omega^2}\kappa^{\otimes 2}((\omega_1,\omega_2),\widetilde A)\,d\nu'(\omega_1,\omega_2)\Big|\\
\le&\sum_{(\omega_1,\omega_2)\in\Omega^2}|\nu(\{(\omega_1,\omega_2)\})-\nu'(\{(\omega_1,\omega_2)\})|\kappa^{\otimes 2}((\omega_1,\omega_2),\widetilde A)\\
\le&M|\Omega|^2.
\end{split}
\end{equation*}
\end{proof}

Now we are ready for the proof of \Cref{thm:existence}.

\begin{proof}[Proof of \Cref{thm:existence}]
We fix a convergent sequence $(\pi_n,\mu_n)_{n=1}^\infty$ of \sqg s.
By \Cref{def:convergence}, for every $k\in\N$, there is a compact set $S_k$ which is the limit of the sequence $\big(\S_k(\pi_n,\mu_n)\big)_{n=1}^\infty$ in the Vietoris topology of the space $\K_k$.
We fix a countable dense subset $D_k$ of $S_k$, $k\in\N$, and we put $D=\bigcup_{k\in\N}D_k$.

The main idea of the proof is as follows.
Let $\I$ be the set of all non-empty finite subsets of $D$.
We will construct a collection $(\rho^I,\nu^I)_{I\in\I}$ of \sqg s which will give rise to two simple inverse systems of measures (which we introduced in \Cref{sec:inverse}):
\begin{itemize}
\item the system given by the probability measures $\rho^I$, $I\in\I$,
\item the system given by the finite measures $\nu^I$, $I\in\I$.
\end{itemize}
We will then apply \Cref{lem:existenceInverseLimit} to deduce that both these simple inverse systems have inverse limits.
The two inverse limits will give rise to a \sqg\ $(\pi,\mu)$.
It will follow from the construction that, for every $k\in\N$, the $k$-shape of $(\pi,\mu)$ contains the set $D_k$.
We will also show that the $k$-shapes of $(\pi,\mu)$ are not much bigger than that, namely, that $\S_k(\pi,\mu)=\overline{D_k}=S_k$, $k\in\N$.
This will prove that the \sqg\ $(\pi,\mu)$ is the desired limit of the sequence $(\pi_n,\mu_n)_{n=1}^\infty$.

\medskip

For every $d=(\rho,\nu)\in D$, let $k_d\in\N$ be such that $d\in D_{k_d}\subseteq S_{k_d}$.
Since $S_{k_d}$ is the limit of the sequence $(\S_{k_d}(\pi_n,\mu_n))_{n=1}^\infty$ in the Vietoris topology of the space $\K_{k_d}$, there is a sequence $(\rho_n,\nu_n)_{n=1}^\infty$ of \sqg s with $(\rho_n,\nu_n)\in\S_{k_d}(\pi_n,\mu_n)$, $n\in\N$, which converges to $d=(\rho,\nu)$ in the space of all \sqg s on $\F_{k_d}$.
As $(\pi_n,\mu_n)^\downarrow_{k_d}$ is a dense subset of $\S_{k_d}(\pi_n,\mu_n)$, $n\in\N$, we may assume that $(\rho_n,\nu_n)\in(\pi_n,\mu_n)^\downarrow_{k_d}$, $n\in\N$.
For every \sqg\ $d=(\rho,\nu)\in D$, we fix such a sequence $(\rho_n,\nu_n)_{n=1}^\infty$ for once and for all.
Further, for every $d=(\rho,\nu)\in D$ and every $n\in\N$, we also fix a morphism $\kappa_n^d$ from $(\pi_n,\mu_n)$ to $(\rho_n,\nu_n)$ for once and for all.

In the following, we denote by $(X_n,\A_n)$ the underlying measurable space of the \sqg\ $(\pi_n,\mu_n)$, $n\in\N$.
We also denote by $\F^I$ the measurable space $\prod_{d\in I}\F_{k_d}$, $I\in\I$.

Now we will construct the collection $(\rho^I,\nu^I)_{I\in\I}$ of \sqg s.
We fix some $I\in\I$ and $n\in\N$.
By \Cref{lem:measurability}, we know that $\prod_{d\in I}\kappa_n^d$ is a Markov kernel from $(X_n,\A_n)$ to $\F^I$.
We define a \sqg\ $(\rho_n^I,\nu_n^I)$ on $\F^I$ by
\[\rho_n^I=\Big(\prod_{d\in I}\kappa_n^d\Big)_*\pi_n\quad\text{and}\quad\nu_n^I=\Big(\prod_{d\in I}\kappa_n^d\Big)^{\otimes 2}_*\mu_n,\]
so that the Markov kernel $\prod_{d\in I}\kappa_n^d$ is a morphism from $(\pi_n,\mu_n)$ to $(\rho_n^I,\nu_n^I)$.
By compactness and the diagonal argument, there is an increasing sequence $(n_p)_{p=1}^\infty$ of natural numbers such that each of the countably many sequences
\[\left(\rho_{n_p}^I(R)\right)_{p=1}^\infty,\quad R\subseteq\prod_{d\in I}[k_d],\quad I\in\I,\]
and
\[\left(\nu_{n_p}^I(\widetilde R)\right)_{p=1}^\infty,\quad\widetilde R\subseteq\Big(\prod_{d\in I}[k_d]\Big)^2,\quad I\in\I,\]
is convergent.
Let us note that it is enough to prove that the subsequence $(\pi_{n_p},\mu_{n_p})_{p=1}^\infty$ has a limit as then the original sequence $(\pi_n,\mu_n)_{n=1}^\infty$, being convergent, has 
necessarily the same limit.
So, without loss of generality, we may assume that $n_p=p$, $p\in\N$, that is, that there is no need to pass to a subsequence.
Then, for every $I\in\I$, we can define the \sqg\ $(\rho^I,\nu^I)$ on $\F^I$ by
\begin{equation}
\label{eq:limitniMiry1}
\rho^I(R)=\lim_{n\to\infty}\rho_n^I(R),\quad R\subseteq\prod_{d\in I}[k_d],
\end{equation}
and
\begin{equation}
\label{eq:limitniMiry2}
\nu^I(\widetilde R)=\lim_{n\to\infty}\nu_n^I(\widetilde R),\quad\widetilde R\subseteq\Big(\prod_{d\in I}[k_d]\Big)^2.
\end{equation}

\medskip

Later, we will show that $(\rho^I)_{I\in\I}$ is a simple inverse system of measures.
Ideally, we would like to know that $(\nu^I)_{I\in\I}$ is also a simple inverse system of measures.
Then the inverse limits of these systems would represent the distributions of vertices and edges, respectively, of the desired limit \sqg.
However, this is not completely correct as one has to distinguish between product spaces of the form $\big(\prod_d[k_d]\big)^2$ and $\prod_d[k_d]^2$.
Indeed, the underlying set of each $\nu^I$ is $\big(\prod_{d\in I}[k_d]\big)^2$.
But to obtain a simple inverse system indexed by $\I$, we need measures whose underlying sets are $\prod_{d\in I}[k_d]^2$, $I\in\I$.
To deal with this small technical issue, let
\[
\theta\colon\Big(\prod_{d\in D}[k_d]\Big)^2\to\prod_{d\in D}[k_d]^2
\]
be the obvious bijection.
Similarly, for every $I\in\I$, let
\[\theta^I\colon\Big(\prod_{d\in\I}[k_d]\Big)^2\to\prod_{d\in\I}[k_d]^2\]
be the obvious bijection.
Further, for every $I\in\I$, let $P_I$ be the canonical projection from $\prod_{d\in D}[k_d]$ to $\prod_{d\in I}[k_d]$, and let $\widetilde P_I$ be the canonical projection from $\prod_{d\in D}[k_d]^2$ to $\prod_{d\in I}[k_d]^2$.
Similarly, for every $I,J\in\I$ with $I\subseteq J$, let $P_{I,J}$ be the canonical projection from $\prod_{d\in J}[k_d]$ to $\prod_{d\in I}[k_d]$, and let $\widetilde P_{I,J}$ be the canonical projection from $\prod_{d\in J}[k_d]^2$ to $\prod_{d\in I}[k_d]^2$.
We note that
\begin{equation}
\label{eq:jednadvesouradnice2}
\theta^I\circ P_I^{\otimes2}=\widetilde P_I\circ\theta,\quad I\in\I,
\end{equation}
and
\begin{equation}
\label{eq:jednadvesouradnice}
\theta^I\circ P_{I,J}^{\otimes2}=\widetilde P_{I,J}\circ\theta^J,\quad I,J\in\I, I\subseteq J.
\end{equation}

\medskip

Now we will show that $(\rho^I)_{I\in\I}$ and $(\theta^I_*\nu^I)_{I\in\I}$ are simple inverse systems of measures.
We need to verify that
\begin{equation}
\label{eq:jednajedna}
\rho^I=(P_{I,J})_*\rho^J,\quad I,J\in\I, I\subseteq J,
\end{equation}
and
\begin{equation}
\label{eq:dvadva}
\theta^I_*\nu^I=(\widetilde P_{I,J})_*(\theta^J_*\nu^J),\quad I,J\in\I, I\subseteq J.
\end{equation}
So we fix $I,J\in\I$ with $I\subseteq J$.
For every $n\in\N$, by \Cref{lem:projekceJakoMorfismy}, the Markov kernel $\kappa_{P_{I,J}}$ is a morphism from $(\rho_n^J,\nu_n^J)$ to $(\rho_n^I,\nu_n^I)$.
In particular, this means that
\begin{equation}
\label{eq:cons}
\rho^I_n=(P_{I,J})_*\rho^J_n\quad\text{and}\quad\nu^I_n=(P_{I,J})^{\otimes 2}_*\nu^J_n,\quad n\in\N.
\end{equation}
Consequently, for every $R\subseteq\prod_{d\in I}[k_d]$, it holds that
\begin{equation*}
\begin{split}
\rho^I(R)\stackrel{\eqref{eq:limitniMiry1}}{=}&\lim_{n\to\infty}\rho^I_n(R)\\
\stackrel{\eqref{eq:cons}}{=}&\lim_{n\to\infty}(P_{I,J})_*\rho^J_n(R)\\
=&\lim_{n\to\infty}\rho^J_n\big(P_{I,J}^{-1}(R)\big)\\
\stackrel{\eqref{eq:limitniMiry1}}{=}&\rho^J\big(P_{I,J}^{-1}(R)\big),
\end{split}
\end{equation*}
which verifies~\eqref{eq:jednajedna}.
Similarly, for every $\widetilde R\subseteq\prod_{d\in I}[k_d]^2$, it holds that
\begin{equation*}
\begin{split}
\theta^I_*\nu^I(\widetilde R)=&\nu^I\big((\theta^I)^{-1}(\widetilde R)\big)\\
\stackrel{\eqref{eq:limitniMiry2}}{=}&\lim_{n\to\infty}\nu^I_n\big((\theta^I)^{-1}(\widetilde R)\big)\\
\stackrel{\eqref{eq:cons}}{=}&\lim_{n\to\infty}(P_{I,J})^{\otimes2}_*\nu^J_n\big((\theta^I)^{-1}(\widetilde R)\big)\\
=&\lim_{n\to\infty}\nu^J_n\big((\theta^I\circ P_{I,J}^{\otimes2})^{-1}(\widetilde R)\big)\\
\stackrel{\eqref{eq:limitniMiry2}}{=}&\nu^J\big((\theta^I\circ P_{I,J}^{\otimes2})^{-1}(\widetilde R)\big)\\
\stackrel{\eqref{eq:jednadvesouradnice}}{=}&\nu^J\big((\widetilde P_{I,J}\circ\theta^J)^{-1}(\widetilde R)\big),
\end{split}
\end{equation*}
which verifies~\eqref{eq:dvadva}.

\medskip

By \Cref{lem:existenceInverseLimit}, both simple inverse systems $(\rho^I)_{I\in\I}$ and $(\theta^I_*\nu^I)_{I\in\I}$ have inverse limits.
Let $\pi$ be the inverse limit of $(\rho^I)_{I\in\I}$, and let $\mu'$ be the inverse limit of $(\theta^I_*\nu^I)_{I\in\I}$.
In particular, $\pi$ is a finite measure on the measurable space $\prod_{d\in D}\F_{k_d}$, and $\mu'$ is a finite measure on the measurable space $\prod_{d\in D}\F_{k_d}^2$.
In fact, $\pi$ is a probability measure as each of the probability measures $\rho^I$, $I\in\I$, is the pushforward of $\pi$ (along the corresponding projection).
We define a measure $\mu$ on the measurable space $\big(\prod_{d\in D}\F_{k_d}\big)^2$ by
\begin{equation}
\label{defmu}
\mu=\theta^{-1}_*\mu'.
\end{equation}
Then $(\pi,\mu)$ is a \sqg\ on $\prod_{d\in D}\F_{k_d}$.

\medskip

By the definition of inverse limits of simple inverse systems, for every $I\in\I$, it holds that
\begin{equation}
\label{obraz1}
\rho^I=(P_I)_*\pi
\end{equation}
and
\begin{equation}
\label{obraz2}
\theta^I_*\nu^I=(\widetilde P_I)_*\mu'.
\end{equation}
For every $I\in\I$ and every $\widetilde R\subseteq\big(\prod_{d\in I}[k_d]\big)^2$ we further have
\begin{equation}
\label{jetomorphism}
\begin{split}
\nu^I(\widetilde R)=&\theta^I_*\nu^I\big(\theta^I(\widetilde R)\big)\\
\stackrel{\eqref{obraz2}}{=}&(\widetilde P_I)_*\mu'\big(\theta^I(\widetilde R)\big)\\
=&\theta^{-1}_*\mu'\big(\theta^{-1}(\widetilde P_I^{-1}(\theta^I(\widetilde R)))\big)\\
\stackrel{\eqref{defmu}}{=}&\mu\big(\theta^{-1}(\widetilde P_I^{-1}(\theta^I(\widetilde R)))\big)\\
=&\mu\big((\widetilde P_I\circ\theta)^{-1}(\theta^I(\widetilde R))\big)\\
\stackrel{\eqref{eq:jednadvesouradnice2}}{=}&\mu\big((\theta^I\circ P_I^{\otimes2})^{-1}(\theta^I(\widetilde R))\big)\\
=&\mu\big((P_I^{\otimes2})^{-1}(\widetilde R)\big)\\
=&(P_I)^{\otimes2}_*\mu(\widetilde R).
\end{split}
\end{equation}
By~\eqref{obraz1}, \eqref{jetomorphism} and~\Cref{ex:deterministicMorphism}, the Markov kernel $\kappa_{P_I}$ is a morphism from $(\pi,\mu)$ to $(\rho^I,\nu^I)$, $I\in\I$.

\medskip

In the rest of the proof, we will show that $(\pi,\mu)$ is a limit of the sequence $(\pi_n,\mu_n)_{n\in\N}$.
That is, we will show that
\[\S_k(\pi,\mu)=S_k,\quad k\in\N.\]
To this end, for the rest of the proof, we fix $k\in\N$.
We also fix $d=(\rho,\nu)\in D_k\subseteq D$ for a while.
We recall that
\begin{equation}
\label{eq:limitniMiry3}
\rho_n^{\{d\}}=(\kappa^d_n)_\star\pi_n=\rho_n,\quad n\in\N,
\end{equation}
and
\begin{equation}
\label{eq:limitniMiry4}
\nu_n^{\{d\}}=(\kappa^d_n)^{\otimes 2}_\star\mu_n=\nu_n,\quad n\in\N.
\end{equation}
Since the sequence $(\rho_n,\nu_n)_{n=1}^\infty$ converges to $d=(\rho,\nu)$ in the space of all \sqg s on $\F_{k_d}$, we obtain from~\eqref{eq:limitniMiry1}, \eqref{eq:limitniMiry2}, \eqref{eq:limitniMiry3} and~\eqref{eq:limitniMiry4} that
\begin{equation}
\label{eq:drhonu}
d=(\rho^{\{d\}},\nu^{\{d\}}).
\end{equation}
We already know that the Markov kernel $\kappa_{P_{\{d\}}}$ is a morphism from $(\pi,\mu)$ to $(\rho^{\{d\}},\nu^{\{d\}})$, and so it follows from~\eqref{eq:drhonu} that $d\in(\pi,\mu)^\downarrow_k$.
As this is true for every $d\in D_k$, we conclude that
\[S_k=\overline{D_k}\subseteq\overline{(\pi,\mu)^\downarrow_k}=\S_k(\pi,\mu).\]
To complete the proof, it only remains to show the reverse inclusion.
But before we continue, we need a little detour.

\medskip

For every $J\in\I$ and every $\alpha=(\alpha_d)_{d\in J}\in\prod_{d\in J}[k_d]$, we put
\[U_\alpha=\Big\{(\beta_d)_{d\in D}\in\prod_{d\in D}[k_d]:\beta_d=\alpha_d\text{ whenever }d\in J\Big\}.\]

\begin{claim}
\label{claim}
Let $\gamma$ be a finite measure on $\prod_{d\in D}\F_{k_d}$.
Let $\psi_m\colon\prod_{d\in D}[k_d]\to[0,1]$, $m=1,\ldots,k$, be measurable functions with $\sum_{m=1}^k\psi_m\equiv 1$.
Let $\eta>0$ be given.
Then there are $J\in\I$ and functions $\psi'_m\colon\prod_{d\in D}[k_d]\to[0,1]$, $m=1,\ldots,k$, with $\sum_{m=1}^k\psi'_m\equiv 1$ such that, for every $m=1,\ldots,k$, we have
\begin{itemize}
\item for every $\alpha\in\prod_{d\in J}[k_d]$, the function $\psi'_m$ is constant on $U_\alpha$,
\item $\|\psi_m-\psi'_m\|_{L^1(\gamma)}<\eta$.
\end{itemize}
\end{claim}

\begin{proof}
We put
\[W=\gamma\big(\prod_{d\in D}[k_d]\big).\]
The assertion is trivial if $W=0$, so we may assume that $W>0$.
	
We equip the set $\prod_{d\in D}[k_d]$ with the product topology (where we consider the discrete topology on each of the finite sets $[k_d]$, $d\in D$); then $\prod_{d\in D}[k_d]$ becomes a metrizable compact space.
The sets $U_\alpha$, where $\alpha\in\prod_{d\in I}[k_d]$ and $I\in\I$, form a base of the topology.
As the topology generates the product $\sigma$-algebra $\prod_{d\in D}2^{[k_d]}$, by~\cite[Theorem~7.8]{Folland}, every finite measure on $\prod_{d\in D}\F_{k_d}$ is regular.
In particular, the measure $\gamma$ is regular.

By Lusin's Theorem (see, e.g., \cite[Theorem~7.10]{Folland}) and regularity of $\gamma$, there are a compact set $K\subseteq\prod_{d\in D}[k_d]$ and continuous functions $\psi_m^1\colon\prod_{d\in D}[k_d]\to[0,1]$, $m=1,\ldots,k$, such that
\[\gamma\big(\prod_{d\in D}[k_d]\setminus K\big)<\frac\eta{2k},\]
and such that
\[\psi_m(\beta)=\psi_m^1(\beta)\quad\text{whenever}\quad\beta\in K,\quad m=1,\ldots,k.\]
Then, for every $\beta\in K$, we have
\[\sum_{m=1}^k\psi^1_m(\beta)=\sum_{m=1}^k\psi_m(\beta)=1.\]
By an easy modification of the functions $\psi^1_m$, $m=1,\ldots,k$, on the set $\prod_{d\in D}[k_d]\setminus K$, we may assure that $\sum_{m=1}^k\psi^1_m\equiv1$.

We put
\[\delta=\frac\eta{2kW}.\]
By continuity of the functions $\psi^1_m$, $m=1,\ldots,k$, every element of $\prod_{d\in D}[k_d]$ is contained in some set $U_\alpha$ (for some $\alpha\in\prod_{d\in I}[k_d]$ and $I\in\I$) such that, for every $m=1,\ldots,k$, each two values of $\psi_m^1$ on $U_\alpha$ differ by at most $\delta$.
As the space $\prod_{d\in D}[k_d]$ is compact, it can be covered by finitely many such sets, let us say by $U_{\alpha_1},\ldots,U_{\alpha_r}$.
For every $j=1,\ldots,r$, let $I_j\in\I$ be such that $\alpha_j\in\prod_{d\in I_j}[k_d]$.
We put $J=\cup_{j=1}^rI_j$.
Then it is easy to verify that the cover
\begin{equation}
\label{cover}
\Big\{U_\alpha:\quad\alpha\in\prod_{d\in J}[k_d]\Big\}
\end{equation}
(consisting of sets which are open and compact at the same time) of $\prod_{d\in D}[k_d]$ is actually a refinement of the cover $\big\{U_{\alpha_1},\ldots,U_{\alpha_r}\big\}$.
Consequently, for every $\alpha\in\prod_{d\in J}[k_d]$ and every $m=1,\ldots,k$, each two values of $\psi_m^1$ on $U_\alpha$ differ by at most $\delta$.
For every $m=1,\ldots,k$, let $\psi_m^2\colon\prod_{d\in D}[k_d]\to[0,1]$ be the function given by
\[\psi_m^2(\beta)=\min_{U_\alpha}\psi_m^1,\quad\beta\in U_\alpha, \alpha\in\prod_{d\in J}[k_d].\]
Finally, we define the functions 
$\psi_m'\colon\prod_{d\in D}[k_d]\to[0,1]$, $m=1,\ldots,k$, by
\[\psi_m'=\psi^2_m,\quad m=1,\ldots,k-1,\]
and
\[\psi_k'=1-\sum_{m=1}^{k-1}\psi^2_m;\]
the fact that all values of $\psi_k'$ belong to $[0,1]$ follows by the obvious fact that
\[0\le\sum_{m=1}^{k-1}\psi^2_m\le\sum_{m=1}^{k-1}\psi^1_m\le 1.\]
Then the functions $\psi_m'$, $m=1,\ldots,k$, sum to $1$ and each of them is constant on each set from the collection~\eqref{cover}.
We clearly have
\[\|\psi_m^2-\psi^1_m\|_{L^\infty(\gamma)}\le\delta,\quad m=1,\ldots,k-1.\]
It follows that, for every $m=1,\ldots,k-1$, we have
\begin{equation}
\label{eq:prvnich-k-1}
\begin{split}
\|\psi_m-\psi'_m\|_{L^1(\gamma)}\le&\|\psi_m-\psi_m^1\|_{L^1(\gamma)}+\|\psi_m^1-\psi'_m\|_{L^1(\gamma)}\\
<&\gamma\big(\prod_{d\in D}[k_d]\setminus K\big)+\|\psi_m^1-\psi^2_m\|_{L^\infty(\gamma)}\gamma\big(\prod_{d\in D}[k_d]\big)\\
\le&\frac\eta{2k}+W\delta\\
=&\frac\eta k.
\end{split}
\end{equation}
Finally, it holds that
\begin{equation*}
\begin{split}
\|\psi_k-\psi'_k\|_{L^1(\gamma)}=&\big\|\big(1-\sum_{m=1}^{k-1}\psi_m\big)-\big(1-\sum_{m=1}^{k-1}\psi'_m\big)\big\|_{L^1(\gamma)}\\
\le&\sum_{m=1}^{k-1}\|\psi_m-\psi'_m\|_{L^1(\gamma)}\\
\stackrel{\eqref{eq:prvnich-k-1}}{<}&\eta.
\end{split}
\end{equation*}
This concludes the proof of the claim.
\end{proof}

The basic structure of the rest of the proof is captured in \Cref{diagram}.

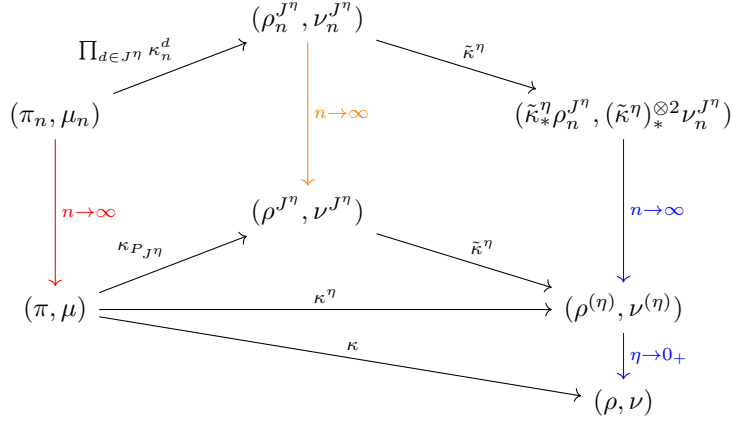
\begin{figure}
\centering
\begin{tikzcd}[column sep=huge]
& (\rho_n^{J^\eta},\nu_n^{J^\eta})
\arrow[dr,"\tilde\kappa^\eta"]
\arrow[dd,orange,"n\to\infty"] & \\
(\pi_n,\mu_n)
\arrow[dd,red,"n\to\infty"]
\arrow[ur,"\prod_{d\in J^\eta}\kappa^d_n"] & & (\tilde\kappa^\eta_*\rho_n^{J^\eta},(\tilde\kappa^\eta)^{\otimes 2}_*\nu_n^{J^\eta})
\arrow[dd,blue,"n\to\infty"] \\
& (\rho^{J^\eta},\nu^{J^\eta})
\arrow[dr,"\tilde\kappa^\eta"] & \\
(\pi,\mu)
\arrow[ur,"\kappa_{P_{J^\eta}}"]
\arrow[rr,"\kappa^\eta"] \arrow[drr,"\kappa"] & & (\rho^{(\eta)},\nu^{(\eta)})
\arrow[d,blue,"\eta\to0_+"] \\
& & (\rho,\nu)
\end{tikzcd}
\caption{
The red arrow represents the convergence which we want to prove.
Black arrows represent morphisms in the category of all \sqg s.
The orange arrow represents convergence in the sense of~\eqref{eq:limitniMiry1} and~\eqref{eq:limitniMiry2}.
Blue arrows represent convergence in the space of all \sqg s on $\F_k$.
}
\label{diagram}
\end{figure}

\medskip

As we already explained, it remains to show that $\S_k(\pi,\mu)\subseteq S_k$
(where $k$ still remains fixed).
In fact, it suffices to show that $(\pi,\mu)^\downarrow_k\subseteq S_k$, because $\S_k(\pi,\mu)$ is the topological closure of $(\pi,\mu)^\downarrow_k$ and the set $S_k$ is closed.
So we fix $(\rho,\nu)\in(\pi,\mu)^\downarrow_k$.
Let $\kappa$ be a morphism from $(\pi,\mu)$ to $(\rho,\nu)$.
Let $P^j\colon\big(\prod_{d\in D}[k_d]\big)^2\to\prod_{d\in D}[k_d]$ be the projection on the $j$th coordinate, $j=1,2$.
Let $\gamma$ be the measure on $\prod_{d\in D}\F_{k_d}$ given by
\begin{equation*}
\gamma=\pi+P^1_*\mu+P^2_*\mu.
\end{equation*}
We fix $\eta>0$ arbitrarily small.
By \Cref{claim} applied to the functions
\[
\kappa(\cdot,\{m\}),\quad m=1,\ldots,k,
\]
we obtain $J^\eta\in\I$ and functions $\psi^\eta_m\colon\prod_{d\in D}[k_d]\to[0,1]$, $m=1,\ldots,k$, which sum to $1$, such that, for every $m=1,\ldots,k$, we have
\begin{itemize}
\item for every $\alpha\in\prod_{d\in J^\eta}[k_d]$, the function $\psi^\eta_m$ is constant on $U_\alpha$,
\item $\|\kappa(\cdot,\{m\})-\psi^\eta_m\|_{L^1(\gamma)}<\eta$.
\end{itemize}
Let $\kappa^\eta$ be the Markov kernel from $\prod_{d\in D}\F_{k_d}$ to $\F_k$ given by
\[
\kappa^\eta(u,F)=\sum_{m\in F}\psi_m^\eta(u),\quad u\in\prod_{d\in D}[k_d],F\subseteq[k].
\]
Then, for every $m=1,\ldots,k$, we have
\begin{equation}
\label{eq:rozdilproostatni}
\|\kappa(\cdot,\{m\})-\kappa^\eta(\cdot,\{m\})\|_{L^1(\gamma)}=\|\kappa(\cdot,\{m\})-\psi^\eta_m\|_{L^1(\gamma)}<\eta.
\end{equation}
Let $(\rho^{(\eta)},\nu^{(\eta)})$ be the \sqg\ on $\F_k$ given by
\begin{equation}
\label{eins}
\rho^{(\eta)}=\kappa^\eta_*\pi\quad\text{and}\quad\nu^{(\eta)}=(\kappa^\eta)^{\otimes 2}_*\mu.
\end{equation}
Then~\eqref{eq:rozdilproostatni} and \Cref{lem:malaZmenaMorphismu} immediately imply that $(\rho^{(\eta)},\nu^{(\eta)})\to(\rho,\nu)$ in the space of all \sqg s on $\F_k$ as $\eta>0$ tends to $0$.
We will show that $(\rho^{(\eta)},\nu^{(\eta)})\in S_k$ for every $\eta>0$.
Since $S_k$ is a closed set, it will follow that $(\rho,\nu)\in S_k$, which will conclude the proof.

So we fix $\eta>0$ for the rest of the proof.
We recall that the Markov kernel $\kappa^\eta$ is a morphism from $(\pi,\mu)$ to $(\rho^{(\eta)},\nu^{(\eta)})$ (by~\eqref{eins}) and that the Markov kernel $\kappa_{P_{J^\eta}}$ is a morphism from $(\pi,\mu)$ to $(\rho^{J^\eta},\nu^{J^\eta})$.
We will use these facts to find a morphism from $(\rho^{J^\eta},\nu^{J^\eta})$ to $(\rho^{(\eta)},\nu^{(\eta)})$.
By our definition of $\kappa^\eta$, for every $\alpha\in\prod_{d\in J^\eta}[k_d]$ and every $F\subseteq[k]$, the function $\kappa^\eta(\cdot,F)$ is constant on $U_\alpha$.
In other words, for every $\alpha\in\prod_{d\in J^\eta}[k_d]$, the probability measure $\kappa^\eta(u,\cdot)$ on $\F_k$ does not depend on the choice of $u\in U_\alpha$; let us denote this probability measure by $p^\eta_\alpha$.
Let $\tilde\kappa^\eta$ be the Markov kernel from $\F^{J^\eta}$ to $\F_k$ given by
\[
\tilde\kappa^\eta(\alpha,F)=p^\eta_\alpha(F),\quad\alpha\in\prod_{d\in J^\eta}[k_d],F\subseteq[k].
\]
Then, for every $F\subseteq[k]$, we have
\begin{equation}
\label{podminka1}
\begin{split}
\rho^{(\eta)}(F)=&\kappa^\eta_*\pi(F)\\
=&\int_{\prod_{d\in D}[k_d]}\kappa^\eta(u,F)\,d\pi(u)\\
=&\sum_{\alpha\in\prod_{d\in J^\eta}[k_d]}\pi(U_\alpha)p^\eta_\alpha(F)\\
=&\sum_{\alpha\in\prod_{d\in J^\eta}[k_d]}\rho^{J^\eta}(\{\alpha\})\tilde\kappa^\eta(\alpha,F)\\
=&\tilde\kappa^\eta_*\rho^{J^\eta}(F).
\end{split}
\end{equation}
Similarly, for every $\widetilde F\subseteq[k]^2$, it holds that
\begin{equation}
\label{podminka2}
\begin{split}
\nu^{(\eta)}(\widetilde F)=&(\kappa^\eta)^{\otimes 2}_*\mu(\widetilde F)\\
=&\int_{\big(\prod_{d\in D}[k_d]\big)^2}(\kappa^\eta)^{\otimes 2}((u_1,u_2),\widetilde F)\,d\mu(u_1,u_2)\\
=&\sum_{(\alpha_1,\alpha_2)\in\big(\prod_{d\in J^\eta}[k_d]\big)^2}\mu(U_{\alpha_1}\times U_{\alpha_2})(p^\eta_{\alpha_1}\times p^\eta_{\alpha_2})(\widetilde F)\\
=&\sum_{(\alpha_1,\alpha_2)\in\big(\prod_{d\in J^\eta}[k_d]\big)^2}\nu^{J^\eta}(\{(\alpha_1,\alpha_2)\})(\tilde\kappa^\eta(\alpha_1,\cdot)\times\tilde\kappa^\eta(\alpha_2,\cdot))(\widetilde F)\\
=&(\tilde\kappa^\eta)^{\otimes 2}_*\nu^{J^\eta}(\widetilde F).
\end{split}
\end{equation}
By~\eqref{podminka1} and~\eqref{podminka2}, we obtain that $\tilde\kappa^\eta$ is a morphism from $(\rho^{J^\eta},\nu^{J^\eta})$ to $(\rho^{(\eta)},\nu^{(\eta)})$.

Let us recall that the \sqg\ $(\rho^{J^\eta},\nu^{J^\eta})$ was defined from the sequence $(\rho^{J^\eta}_n,\nu^{J^\eta}_n)_{n=1}^\infty$ as in~\eqref{eq:limitniMiry1} and~\eqref{eq:limitniMiry2}.
Thus, \Cref{lem:malaZmenaGraphonu} easily implies that the sequence $\big(\tilde\kappa^\eta_*\rho_n^{J^\eta},(\tilde\kappa^\eta)^{\otimes 2}_*\nu_n^{J^\eta}\big)_{n=1}^\infty$ of \sqg s converges to $(\rho^{(\eta)},\nu^{(\eta)})=(\tilde\kappa_*^\eta\rho^{J^\eta},(\tilde\kappa^\eta)^{\otimes 2}_*\nu^{J^\eta})$ in the space of all \sqg s on $\F_k$.
For every $n\in\N$, the Markov kernel $\tilde\kappa^\eta\circ\big(\prod_{d\in J^\eta}\kappa^d_n\big)$ is a morphism from $(\pi_n,\mu_n)$ to $(\tilde\kappa^\eta_*\rho_n^{J^\eta},(\tilde\kappa^\eta)^{\otimes 2}_*\nu_n^{J^\eta})$ (see \Cref{diagram}).
In particular, \[(\tilde\kappa^\eta_*\rho_n^{J^\eta},(\tilde\kappa^\eta)^{\otimes 2}_*\nu_n^{J^\eta})\in(\pi_n,\mu_n)^\downarrow_k\subseteq\S_k(\pi_n,\mu_n),\quad n\in\N.\]
Consequently, as the sequence $(\S_k(\pi_n,\mu_n))_{n=1}^\infty$ of compact sets converges to $S_k$, we easily obtain that $(\rho^{(\eta)},\nu^{(\eta)})\in S_k$.
This completes the proof.
\end{proof}

\begin{remark}
\label{rem:nonatomic}
The underlying measurable space of the limit \sqg\ $(\pi,\mu)$ constructed in the proof of \Cref{thm:existence} is the countable product of finite spaces.
In particular, all its singleton subsets are measurable.
By \Cref{ex:nonatomic}, there is a \sqg\ $(\pi',\mu')$ on a standard Borel measurable space such that the measure $\pi'$ is continuous (that is, all singletons have $\pi'$-measure zero) and such that there are morphisms between $(\pi,\mu)$ and $(\pi',\mu')$ in both directions.
Then $(\pi,\mu)$ and $(\pi',\mu')$ have the same $k$-shapes, and so $(\pi',\mu')$ is also a limit of the convergent sequence $(\pi_n,\mu_n)_{n=1}^\infty$.

This shows that the limit \sqg\ of a given convergent sequence can be always chosen in such a way that its underlying measurable space is standard Borel and such that the probability measure describing the distribution of vertices is continuous.
Consequently, by an application of the isomorphism theorem for measures (see e.g.~\cite[Theorem~17.41]{Kechris}), we can even find the limit \sqg\ in such a way that its underlying measurable space is the unit interval and that the distribution of vertices is the Lebesgue measure.
\end{remark}

\section{Comparison with s-convergence}
\label{sec:s-graphons}

The notion of convergence of \sqg s from \Cref{def:convergence} is inspired by s-convergence of graph sequences introduced in~\cite{KLS}.
Both these convergence notions are defined as convergence of certain $k$-shapes in the Vietoris topology;
the only difference is hidden in the definition of the $k$-shapes.
In \Cref{def:k-shape}, we defined the $k$-shape $\S_k(\pi,\mu)$ of a \sqg\ $(\pi,\mu)$ as the topological closure of the set
\[(\pi,\mu)^\downarrow_k=\big\{(\rho,\nu)\in(\pi,\mu)^\downarrow:(\rho,\nu)\text{ is a \sqg\ on $\F_k$}\big\}.\]
Alternatively, in the spirit of~\cite{KLS}, we could require all elements of the $k$-shape to have the same fixed distribution of vertices, namely the normalized counting measure on $\F_k$, which we denote by $\gamma_k$.
To realize this idea, for every $k\in\N$, let us consider the space of all finite measures on $\F_k^2$ as a topological subspace of $\R^{[k]^2}$, and let us define the $k$-shape $\widetilde\S_k(\pi,\mu)$ of a \sqg\ $(\pi,\mu)$ as the topological closure of the set
\[\big\{\nu:(\gamma_k,\nu)\in(\pi,\mu)^\downarrow_k\big\}.\]
Up to a simple identification of non-negative $k$-by-$k$ matrices with finite measures on $\F_k^2$, this is a straightforward generalization of the $k$-shapes (of finite graphs/s-graphons) introduced in~\cite{KLS}.
The main goal of this section is to show that, no matter what definition of $k$-shapes we use (either $\S_k(\cdot)$ or $\widetilde\S_k(\cdot)$), we obtain the same convergence notion.

\medskip

In the following, for any non-empty finite set $F$, we denote by $d_\infty^F$ the maximum metric on $\R^F$.
We identify finite measures on the measurable space $(F,2^F)$ with the corresponding elements of $\R^F$, so that we can measure the $d_\infty^F$-distance of two such measures.
We also identify \sqg s on $\F_k$ with the corresponding elements of $\R^{[k]\cup[k]^2}$, so that we can measure the $d_\infty^{[k]\cup[k]^2}$-distance of two such \sqg s, $k\in\N$.

\begin{lemma}
\label{lem:zmenaPrsti}
Suppose that $(\pi,\mu)$ is a \sqg\ on a measurable space $(X,\A)$, $k\in\N$ and $(\rho,\nu)\in(\pi,\mu)^\downarrow_k$.
Suppose that
\begin{equation}
\label{eq:min2}
d:=d_\infty^{[k]}(\rho,\gamma_k)\le\frac1{4k^2}.
\end{equation}
Then there is a finite measure $\nu'$ on $\F_k^2$ such that $(\gamma_k,\nu')\in(\pi,\mu)^\downarrow_k$ and such that
\[
d_\infty^{[k]^2}(\nu,\nu')\le 4(1+2\mu(X^2))k(k-1)d.
\]
\end{lemma}

\begin{proof}
The idea of the proof is taken from the proof of~\cite[Lemma~3.1]{KLS}.

Let $\kappa$ be a morphism from $(\pi,\mu)$ to $(\rho,\nu)$.
Inductively, we will modify the Markov kernel $\kappa$ (in finitely many steps) to obtain a new Markov kernel $\kappa'$ from $(X,\A)$ to $\F_k$ such that (among other properties) it holds that $\kappa'_*\pi=\gamma_k$.
As the first step, we put $\kappa^1=\kappa$.
Now suppose that, for some $j\in\N$, we have already constructed a Markov kernel $\kappa^j$ from $(X,\A)$ to $\F_k$ such that
\begin{equation}
\label{eq:maxDist}
d_\infty^{[k]}(\kappa^j_*\pi,\gamma_k)\le d,
\end{equation}
and such that
\begin{equation}
\label{eq:o1vic}
\big|\big\{m\in[k]:\kappa^j_*\pi(\{m\})=\frac1k\big\}\big|\ge j-1.
\end{equation}
If $\kappa^j_*\pi=\gamma_k$ then we put $\kappa'=\kappa^j$ and stop the construction.
Otherwise, as $\kappa^j_*\pi$ is a probability measure,
there are $m_1,m_2\in[k]$ such that
\begin{equation}
\label{eq:sandwich}
\kappa^j_*\pi(\{m_1\})>\frac1k\quad\text{and}\quad\kappa^j_*\pi(\{m_2\})<\frac1k.
\end{equation}
Then it holds that
\begin{equation}
\label{eq:druhyDilek}
\begin{split}
\frac1k\stackrel{\eqref{eq:sandwich}}{<}&\kappa^j_*\pi(\{m_1\})\\
=&\int_{\big\{x\in X:\kappa^j(x,\{m_1\})\le 2dk\big\}}\kappa^j(x,\{m_1\})\,d\pi(x)\\
&+\int_{\big\{x\in X:\kappa^j(x,\{m_1\})>2dk\big\}}\kappa^j(x,\{m_1\})\,d\pi(x)\\
\le&2dk+\pi\big(\big\{x\in X:\kappa^j(x,\{m_1\})>2dk\big\}\big).
\end{split}
\end{equation}
For every $\beta\ge 0$, let $F_\beta\colon X\to[0,1]$ be the function defined by
\begin{equation}
\label{eq:F_beta}
F_\beta(x)=\max\big\{\kappa^j(x,\{m_1\})-\beta,0\big\},\quad x\in X.
\end{equation}
Then we have
\begin{equation}
\label{eq:shora}
\begin{split}
\int_XF_0(x)\,d\pi(x)=&\int_X\kappa^j(x,\{m_1\})\,d\pi(x)\\
&=\kappa^j_*\pi(\{m_1\})\\
\stackrel{\eqref{eq:sandwich}}{>}&\frac1k
\end{split}
\end{equation}
and
\begin{equation}
\label{eq:zdola}
\begin{split}
\int_XF_{2dk}(x)\,d\pi(x)=&\int_{\big\{x\in X:\kappa^j(x,\{m_1\})>2dk\big\}}\big(\kappa^j(x,\{m_1\})-2dk\big)\,d\pi(x)\\
\le&\int_X\kappa^j(x,\{m_1\})\,d\pi(x)-2dk\pi\big(\big\{x\in X:\kappa^j(x,\{m_1\})>2dk\big\}\big)\\
\stackrel{\eqref{eq:druhyDilek}}{<}&\kappa^j_*\pi(\{m_1\})-2dk\big(\frac1k-2dk\big)\\
\stackrel{\eqref{eq:maxDist}}{\le}&\frac1k-d+4d^2k^2\\
\stackrel{\eqref{eq:min2}}{\le}&\frac1k.
\end{split}
\end{equation}
By the obvious continuity of the map
\begin{equation*}
\beta\mapsto\int_XF_\beta(x)\,d\pi(x),\quad\beta\ge0,
\end{equation*}
it follows from~\eqref{eq:shora} and~\eqref{eq:zdola} that there is some $\beta_0\in(0,2dk)$ such that
\begin{equation}
\label{eq:intF_beta0}
\int_XF_{\beta_0}(x)\,d\pi(x)=\frac1k.
\end{equation}
Let $\kappa^{j+1}$ be the Markov kernel from $(X,\A)$ to $\F_k$ given by
\begin{multline}
\label{indStep}
\kappa^{j+1}(x,\{m\})=
\begin{cases}
\kappa^j(x,\{m\}),&m\in[k]\setminus\{m_1,m_2\},\\
F_{\beta_0}(x),&m=m_1,\\
\kappa^j(x,\{m_1\})+\kappa^j(x,\{m_2\})-F_{\beta_0}(x),&m=m_2,
\end{cases}
\\
x\in X.
\end{multline}
Then it holds that
\begin{equation}
\label{eq:stareM}
\kappa^{j+1}_*\pi(\{m\})=\kappa^j_*\pi(\{m\}),\quad m\in[k]\setminus\{m_1,m_2\},
\end{equation}
as well as
\begin{equation}
\label{eq:m1}
\kappa^{j+1}_*\pi(\{m_1\})=\frac1k
\end{equation}
(by~\eqref{eq:intF_beta0}), and
\begin{equation}
\label{eq:m2}
\begin{split}
&\Big|\kappa^{j+1}_*\pi(m_2)-\frac1k\Big|\\
=&\Big|\int_X\kappa^{j+1}(x,\{m_2\})\,d\pi(x)-\frac1k\Big|\\
=&\Big|\int_X\kappa^j(x,\{m_1\})\,d\pi(x)+\int_X\kappa^j(x,\{m_2\})\,d\pi(x)-\int_XF_{\beta_0}(x)\,d\pi(x)-\frac1k\Big|\\
\stackrel{\eqref{eq:intF_beta0}}{=}&\Big|\Big(\kappa^j_*\pi(\{m_1\})-\frac1k\Big)+\Big(\kappa^j_*\pi(\{m_2\})-\frac1k\Big)\Big|\\
\stackrel{\eqref{eq:maxDist},\eqref{eq:sandwich}}{<}&d.
\end{split}
\end{equation}
By~\eqref{eq:maxDist}, \eqref{eq:stareM}, \eqref{eq:m1} and \eqref{eq:m2}, it follows that
\begin{equation*}
d_\infty^{[k]}(\kappa^{j+1}_*\pi,\gamma_k)\le d.
\end{equation*}
Also, \eqref{eq:o1vic}, \eqref{eq:sandwich}, \eqref{eq:stareM} and \eqref{eq:m1} imply that
\begin{equation*}
\big|\big\{m\in[k]:\kappa^{j+1}_*\pi(\{m\})=\frac1k\big\}\big|\ge j.
\end{equation*}
This finishes the inductive step.
We note that, by~\eqref{eq:F_beta} and~\eqref{indStep}, it holds that
\begin{equation}
\label{eq:zmenaOBeta}
\big|\kappa^{j+1}(x,\{m\})-\kappa^j(x,\{m\})\big|\le\beta_0<2dk,\quad x\in X,m\in[k].
\end{equation}

By~\eqref{eq:o1vic}, since each $\kappa^j_*\pi$ is a probability measure, the construction stops at the $l$th step for some $l\le k$.
Then, as $\kappa=\kappa^1$ and $\kappa'=\kappa^l$, we obtain from~\eqref{eq:zmenaOBeta} that
\begin{equation}
\label{eq:(k-1)beta}
\big|\kappa(x,\{m\})-\kappa'(x,\{m\})\big|\le2k(k-1)d,\quad x\in X,m\in[k].
\end{equation}
Let $P^j\colon X^2\to X$ be the projection on the $j$th coordinate, $j=1,2$, and let $\gamma$ be the measure on $(X,\A)$ given by
\begin{equation*}
\gamma=\pi+P^1_*\mu+P^2_*\mu.
\end{equation*}
Then, by~\eqref{eq:(k-1)beta}, it holds that
\begin{equation*}
\big\|\kappa(\cdot,\{m\})-\kappa'(\cdot,\{m\})\big\|_{L^1(\gamma)}\le2(1+2\mu(X^2))k(k-1)d,\quad m\in[k].
\end{equation*}
So, if we put $\nu'=(\kappa')^{\otimes2}_*\mu$ (so that $(\gamma_k,\nu')\in(\pi,\mu)^\downarrow_k$), then \Cref{lem:malaZmenaMorphismu} implies that
\begin{equation*}
\big|\nu(\{(m_1,m_2)\})-\nu'(\{(m_1,m_2)\})\big|\le4(1+2\mu(X^2))k(k-1)d,\quad (m_1,m_2)\in[k]^2.
\end{equation*}
This completes the proof.
\end{proof}

\begin{corollary}
\label{cor:malyVelkyShapeEasy}
Let $(\pi,\mu)$ be a \sqg\ on a measurable space $(X,\A)$ and let $k\in\N$ be fixed.
Then
\[
\widetilde S_k(\pi,\mu)=\big\{\nu:(\gamma_k,\nu)\in\S_k(\pi,\mu)\big\}.
\]
\end{corollary}

\begin{proof}
Suppose first that $\nu\in\widetilde S_k(\pi,\mu)$.
By definition, that means that there is a sequence $(\nu_n)_{n=1}^\infty$ of measures on $\F_k^2$, with $(\gamma_k,\nu_n)\in(\pi,\mu)^\downarrow_k$, $n\in\N$, which converges to $\nu$.
Consequently, the sequence $(\gamma_k,\nu_n)_{n=1}^\infty$ of \sqg s converges to $(\gamma_k,\nu)$, and so $(\gamma_k,\nu)\in\S_k(\pi,\mu)$.
This proves the inclusion `$\subseteq$'.

Now suppose that $\nu$ is a measure on $\F_k^2$ such that $(\gamma_k,\nu)\in\S_k(\pi,\mu)$.
Then there is a sequence $(\rho_n,\nu_n)_{n=1}^\infty$ of \sqg s, with $(\rho_n,\nu_n)\in(\pi,\mu)^\downarrow_k$, $n\in\N$, which converges to $(\gamma_k,\nu)$.
This means that
\begin{equation}
\label{eq:to0a}
d_\infty^{[k]}(\rho_n,\gamma_k)\to0,\quad n\to\infty
\end{equation}
and
\begin{equation}
\label{eq:to0ab}
d_\infty^{[k]^2}(\nu_n,\nu)\to0,\quad n\to\infty.
\end{equation}
We may assume that $d_\infty^{[k]}(\rho_n,\gamma_k)\le\frac1{4k^2}$, $n\in\N$.
Then, for every $n\in\N$, we can apply \Cref{lem:zmenaPrsti} to find a finite measure $\nu_n'$ on $\F_k^2$ such that
$(\gamma_k,\nu_n')\in(\pi,\mu)^\downarrow_k$ and such that
\begin{equation}
\label{eq:to0b}
d_\infty^{[k]^2}(\nu_n,\nu_n')\le4(1+2\mu(X^2))k(k-1)d_\infty^{[k]}(\rho_n,\gamma_k).
\end{equation}
By~\eqref{eq:to0a} and~\eqref{eq:to0b}, it follows that
\[
d_\infty^{[k]^2}(\nu_n,\nu_n')\to0,\quad n\to\infty.
\]
Consequently, by~\eqref{eq:to0ab}, we also have
\[
d_\infty^{[k]^2}(\nu_n',\nu) \to0,\quad n\to\infty.
\]
It follows that $\nu\in\widetilde S_k(\pi,\mu)$, which completes the proof of the inclusion `$\supseteq$'.
\end{proof}

For any non-empty finite set $F$, let $d_H^F$ be the Hausdorff distance on the hyperspace of all non-empty compact subsets of $\R^F$, arising from the maximum metric $d_\infty^F$ on $\R^F$.

\begin{corollary}
\label{cor:spojitost}
Let $(\pi_X,\mu_X)$ and $(\pi_Y,\mu_Y)$ be \sqg s on measurable
spaces $(X,\A)$ and $(Y,\B)$, respectively.
Let $k\in\N$ and suppose that
\begin{equation}
\label{eq:nutno}
d_H^{[k]\cup[k]^2}\big(\S_k(\pi_X,\mu_X),\S_k(\pi_Y,\mu_Y)\big)<\frac1{4k^2}.
\end{equation}
Let
\[U=\max\big\{\mu_X(X^2),\mu_Y(Y^2)\big\}.\]
Then
\begin{multline*}
d_H^{[k]^2}\big(\widetilde\S_k(\pi_X,\mu_X),\widetilde\S_k(\pi_Y,\mu_Y)\big)\\
\le\Big(1+(1+2U)4k(k-1)\Big)d_H^{[k]\cup[k]^2}\big(\S_k(\pi_X,\mu_X),\S_k(\pi_Y,\mu_Y)\big).
\end{multline*}
\end{corollary}

\begin{proof}
By symmetry, it is enough to show that, for every $\nu^1\in\widetilde\S_k(\pi_X,\mu_X)$ and every (small enough) $\varepsilon>0$, there is $\nu^2\in\widetilde\S_k(\pi_Y,\mu_Y)$ such that
\begin{equation*}
d_\infty^{[k]^2}(\nu^1,\nu^2)
\le\Big(1+(1+2U)4k(k-1)\Big)\Big(d_H^{[k]\cup[k]^2}\big(\S_k(\pi_X,\mu_X),\S_k(\pi_Y,\mu_Y)\big)+\varepsilon\Big).
\end{equation*}
So we fix $\nu^1\in\widetilde\S_k(\pi_X,\mu_X)$ and $\varepsilon>0$ such that
\begin{equation}
\label{eq:epsilonMezi}
\varepsilon<\frac1{4k^2}-d_H^{[k]\cup[k]^2}\big(\S_k(\pi_X,\mu_X),\S_k(\pi_Y,\mu_Y)\big)
\end{equation}
(which is possible by~\eqref{eq:nutno}).
By \Cref{cor:malyVelkyShapeEasy}, it holds that $(\gamma_k,\nu^1)\in\S_k(\pi_X,\mu_X)$.
As $(\pi_Y,\mu_Y)^\downarrow_k$ is a dense subset of $\S_k(\pi_Y,\mu_Y)$, we can find $(\overline\rho,\overline\nu)\in(\pi_Y,\mu_Y)^\downarrow_k$ such that
\begin{equation}
\label{eq:distAB}
d_\infty^{[k]\cup[k]^2}\big((\gamma_k,\nu^1),(\overline\rho,\overline\nu)\big)\le d_H^{[k]\cup[k]^2}\big(\S_k(\pi_X,\mu_X),\S_k(\pi_Y,\mu_Y)\big)+\varepsilon.
\end{equation}
Then, in particular, we have
\[
d_\infty^{[k]}(\overline\rho,\gamma_k)\stackrel{\eqref{eq:distAB}}{\le}d_H^{[k]\cup[k]^2}\big(\S_k(\pi_X,\mu_X),\S_k(\pi_Y,\mu_Y)\big)+\varepsilon\stackrel{\eqref{eq:epsilonMezi}}{<}\frac1{4k^2}.
\]
By \Cref{lem:zmenaPrsti}, there is a finite measure $\nu^2$ on $\F_k^2$ such that $(\gamma_k,\nu^2)\in(\pi_Y,\mu_Y)^\downarrow_k$ and such that
\begin{equation}
\label{eq:distB}
d_\infty^{[k]^2}(\overline\nu,\nu^2)\le4(1+2\mu_Y(Y^2))k(k-1)d_\infty^{[k]}(\overline\rho,\gamma_k).
\end{equation}
Then
\begin{equation}
\label{eq:distBB}
d_\infty^{[k]^2}(\overline\nu,\nu^2)\stackrel{\eqref{eq:distAB},\eqref{eq:distB}}{\le}4(1+2\mu_Y(Y^2))k(k-1)\Big(d_H^{[k]\cup[k]^2}\big(\S_k(\pi_X,\mu_X),\S_k(\pi_Y,\mu_Y)\big)+\varepsilon\Big).
\end{equation}
Finally, we have
\begin{multline*}
d_\infty^{[k]^2}(\nu^1,\nu^2)\le d_\infty^{[k]^2}(\nu^1,\overline\nu)+d_\infty^{[k]^2}(\overline\nu,\nu^2)\\
\stackrel{\eqref{eq:distAB},\eqref{eq:distBB}}{\le}\Big(1+4(1+2\mu_Y(Y^2))k(k-1)\Big)\Big(d_H^{[k]\cup[k]^2}\big(\S_k(\pi_X,\mu_X),\S_k(\pi_Y,\mu_Y)\big)+\varepsilon\Big),
\end{multline*}
and the conclusion follows.
\end{proof}

\begin{proposition}
\label{prop:s-shapy}
Let $(\pi_X,\mu_X)$ and $(\pi_Y,\mu_Y)$ be \sqg s on measurable
spaces $(X,\A)$ and $(Y,\B)$, respectively.
Then the following conditions are equivalent:
\begin{itemize}
\item[\textnormal{(i)}] $\widetilde\S_k(\pi_X,\mu_X)=\widetilde\S_k(\pi_Y,\mu_Y),\quad k\in\N,$
\item[\textnormal{(ii)}] $\S_k(\pi_X,\mu_X)=\S_k(\pi_Y,\mu_Y),\quad k\in\N.$
\end{itemize}
\end{proposition}

\begin{proof}
The implication (ii)$\implies$(i) follows by \Cref{cor:malyVelkyShapeEasy}.

To prove the opposite implication, suppose that there is $k\in\N$ such that $\S_k(\pi_X,\mu_X)\neq\S_k(\pi_Y,\mu_Y)$;
let us say that $\S_k(\pi_X,\mu_X)\setminus\S_k(\pi_Y,\mu_Y)\neq\emptyset$.
As $(\pi_X,\mu_X)^\downarrow_k$ is a dense subset of $\S_k(\pi_X,\mu_X)$ and the $k$-shape $\S_k(\pi_Y,\mu_Y)$ is a closed set, we can fix some $(\rho,\nu)\in(\pi_X,\mu_X)^\downarrow_k\setminus\S_k(\pi_Y,\mu_Y)$, together with a morphism $\kappa$ from $(\pi_X,\mu_X)$ to $(\rho,\nu)$.

Next, we construct $(\rho',\nu')\in(\pi_X,\mu_X)^\downarrow_k\setminus\S_k(\pi_Y,\mu_Y)$ such that all values of the probability measure $\rho'$ are positive rational numbers.
This can be achieved by an application of \Cref{lem:malaZmenaMorphismu} as follows.
We fix $\eta\in(0,1)$.
We also fix some $m_0\in[k]$ with $\rho(\{m_0\})>0$.
Further, for every $m\in[k]\setminus\{m_0\}$, we fix some
\[c_m\in\Big(0,\frac{\eta}{k(1+2\mu_X(X^2))}\Big)\]
(then, in particular, $\sum_{m\in[k]\setminus\{m_0\}}c_m<\eta<1$) such that
\begin{equation}
\label{eq:jednaZmena}
\rho(\{m\})+c_m\rho(\{m_0\})\in\Q,\quad m\in[k]\setminus\{m_0\}.
\end{equation}
Then we define a Markov kernel $\kappa'$ from $(X,\A)$ to $\F_k$ by
\begin{multline}
\label{predefKernel}
\kappa'(x,\{m\})=
\begin{cases}
\kappa(x,\{m\})+c_m\kappa(x,\{m_0\}),&m\in[k]\setminus\{m_0\},\\
\big(1-\sum_{m\in[k]\setminus\{m_0\}}c_m\big)\kappa(x,\{m_0\}),&m=m_0,
\end{cases}
\\
x\in X.
\end{multline}
Then it is easy to see that, if $\gamma=\pi+P^1_*\mu_X+P^2_*\mu_X$ (where $P^j\colon X^2\to X$ is the projection on the $j$th coordinate, $j=1,2$), then
\begin{equation*}
\big\|\kappa(\cdot,\{m\})-\kappa'(\cdot,\{m\})\big\|_{L^1(\gamma)}<\eta(1+2\mu_X(X^2)),\quad m\in[k].
\end{equation*}
So, by \Cref{lem:malaZmenaMorphismu}, if we define
\[\rho'=\kappa'_*\pi\quad\text{and}\quad\nu'=(\kappa')^{\otimes2}_*\mu\]
for $\eta$ sufficiently close to $0$, then $(\rho',\nu')\in(\pi,\mu)^\downarrow_k$ does not belong to the closed set $\S_k(\pi_Y,\mu_Y)$.
Also, for every $m\in[k]\setminus\{m_0\}$, we have
\begin{equation*}
\begin{split}
\rho'(\{m\})=&\kappa'_*\pi(\{m\})\\
=&\int_X\kappa'(x,\{m\})\,d\pi(x)\\
\stackrel{\eqref{predefKernel}}{=}&\int_X\kappa(x,\{m\})\,d\pi(x)+c_m\int_X\kappa(x,\{m_0\})\,d\pi(x)\\
=&\rho(\{m\})+c_m\rho(\{m_0\})\stackrel{\eqref{eq:jednaZmena}}{\in}\Q,
\end{split}
\end{equation*}
and so
\begin{equation*}
\begin{split}
\rho'(\{m_0\})=&1-\sum_{m\in[k]\setminus\{m_0\}}\rho'(\{m\})\in\Q.
\end{split}
\end{equation*}
Also, by the definition of $\kappa'$, it is easy to see that all values of $\rho'=\kappa'_*\pi$ are positive.
So the \sqg\ $(\rho',\nu')$ has all the required properties.

As all values of $\rho'$ are positive rational numbers, there are $s\in\N$ and $r_1,\ldots,r_k\in\N$ such that
\[
\rho'(\{m\})=\frac{r_m}s,\quad m\in[k].
\]
We fix a map $f\colon[s]\to[k]$ such that $|f^{-1}(\{m\})|=r_m$, $m\in[k]$.
Let $\overline\kappa$ be the Markov kernel from $\F_k$ to $\F_s$ given by
\[
\overline\kappa(m,F)=\frac1{r_m}\big|F\cap f^{-1}(\{m\})\big|,\quad m\in[k],F\subseteq[s].
\]
We define a \sqg\ $(\rho'',\nu'')$ on $\F_s$ by
\[\rho''=\overline\kappa_*\rho'\quad\text{and}\quad\nu''=\overline\kappa^{\otimes2}_*\nu'.\]
Then, for every $F\subseteq[s]$, it holds that
\begin{equation*}
\begin{split}
\rho''(F)=&\int_{[k]}\overline\kappa(m,F)\,d\rho'(m)\\
=&\frac1s\sum_{m\in[k]}\big|F\cap f^{-1}(\{m\})\big|\\
=&\frac 1s|F|,
\end{split}
\end{equation*}
and so $\rho''=\gamma_s$.
Thus
\[(\gamma_s,\nu'')=(\rho'',\nu'')\preceq(\rho',\nu')\preceq(\pi_X,\mu_X),\]
and so, in particular, $\nu''\in\widetilde\S_s(\pi_X,\mu_X)$.
We will show that $(\gamma_s,\nu'')\notin\S_s(\pi_Y,\mu_Y)$; that will imply that $\nu''\notin\widetilde\S_s(\pi_Y,\mu_Y)$, and so $\widetilde\S_s(\pi_X,\mu_X)\neq\widetilde\S_s(\pi_Y,\mu_Y)$, which will complete the proof.

By \Cref{ex:trivialMorphisms}, the Markov kernel $\kappa_f$ is a morphism from $(\gamma_s,\nu'')$ to $(\rho',\nu')$.
By \Cref{lem:malaZmenaGraphonu}, the map $\phi_f$ from the space of all \sqg s on $\F_s$ to the space of all \sqg s on $\F_k$ given by
\[
\phi_f(\tilde\rho,\tilde\nu)=((\kappa_f)_*\tilde\rho,(\kappa_f)^{\otimes2}_*\tilde\nu)
\]
is continuous.
Obviously, it holds
\[
\phi_f\big((\pi_Y,\mu_Y)^\downarrow_s\big)\subseteq(\pi_Y,\mu_Y)^\downarrow_k.
\]
By continuity, it follows that
\[
\phi_f\big(\S_s(\pi_Y,\mu_Y)\big)\subseteq\S_k(\pi_Y,\mu_Y).
\]
So, as
\[\phi_f(\gamma_s,\nu'')=(\rho',\nu')\notin\S_k(\pi_Y,\mu_Y),\]
we must have $(\gamma_s,\nu'')\notin\S_s(\pi_Y,\mu_Y)$, as we needed.
\end{proof}

Let $\sim$ be the equivalence relation on the class of all \sqg s given by
\[
(\pi_1,\mu_1)\sim(\pi_2,\mu_2)\quad\Leftrightarrow\quad\forall k\in\N\ \big(\S_k(\pi_1,\mu_1)=\S_k(\pi_2,\mu_2)\big).
\]
By \Cref{prop:s-shapy}, the relation $\sim$ can be equivalently described by
\[
(\pi_1,\mu_1)\sim(\pi_2,\mu_2)\quad\Leftrightarrow\quad\forall k\in\N\ \big(\widetilde\S_k(\pi_1,\mu_1)=\widetilde\S_k(\pi_2,\mu_2)\big).
\]
For every \sqg\ $(\pi,\mu)$, let $[(\pi,\mu)]$ denote the $\sim$-equivalence class of $(\pi,\mu)$.
Let $\mathcal G$ be the space of all $\sim$-equivalence classes.
Let us consider the following two embeddings of $\mathcal G$ into the product of hyperspaces of compact sets:
\[[(\pi,\mu)]\mapsto(\S_k(\pi,\mu))_{k\in\N}\quad\text{and}\quad[(\pi,\mu)]\mapsto(\widetilde\S_k(\pi,\mu))_{k\in\N}.\]
Each of the hyperspaces is equipped with the corresponding Vietoris topology, and their products are equipped with the product topologies.
These embeddings provide two topologies on the space $\mathcal G$: the topology of convergence of $k$-shapes $\S_k(\cdot)$ and the topology of convergence of $k$-shapes $\widetilde\S_k(\cdot)$, respectively.
The following theorem states that these two topologies on the space $\mathcal G$ coincide.

\begin{theorem}
\label{thm:2topologie}
Let $(\pi_n,\mu_n)$, $n\in\N$, and $(\pi,\mu)$ be \sqg s on measurable spaces $(X_n,\A_n)$, $n\in\N$, and $(X,\A)$, respectively.
Then the following conditions are equivalent:
\begin{itemize}
\item[\textnormal{(a)}] for every $k\in\N$, the sequence $\big(\S_k(\pi_n,\mu_n)\big)_{n=1}^\infty$ is convergent (and its limit is $\S_k(\pi,\mu)$),
\item[\textnormal{(b)}] for every $k\in\N$, the sequence $\big(\widetilde\S_k(\pi_n,\mu_n)\big)_{n=1}^\infty$ is convergent (and its limit is $\widetilde\S_k(\pi,\mu)$).
\end{itemize}
\end{theorem}

\begin{remark}
By \Cref{def:convergence}, condition \textnormal{(a)} from \Cref{thm:2topologie} can be equivalently reformulated as
\begin{itemize}
\item[\textnormal{(a')}] the sequence $(\pi_n,\mu_n)$ of \sqg s is convergent (and its limit is $(\pi,\mu)$).
\end{itemize}
\end{remark}

\begin{proof}[Proof of \Cref{thm:2topologie}]
If $\sup\{\mu_n(X_n^2):n\in\N\}=\infty$, then neither of the sequences
\[
\big(\S_k(\pi_n,\mu_n)\big)_{n=1}^\infty\quad k\in\N,
\]
and
\[
\big(\widetilde\S_k(\pi_n,\mu_n)\big)_{n=1}^\infty\quad k\in\N,
\]
of compact sets is uniformly bounded (cf. the last paragraph of the proof of \Cref{lem:compactness}).
In particular, neither of these sequences is convergent in the corresponding Vietoris topology, and so neither of the conditions (a) and (b) holds.

So let us assume that 
\[K:=\sup\{\mu_n(X_n^2):n\in\N\}<\infty.\]
Let $\iota$ be the embedding
\[[(\pi',\mu')]\mapsto(\S_k(\pi',\mu'))_{k\in\N}\]
of the space $\mathcal G$ into the product of hyperspaces.
Let $\mathcal G_K\subseteq\mathcal G$ be the space of $\sim$-equivalence classes of all those \sqg s $(\pi',\mu')$ on some measurable space $(X',\A')$ for which $\mu'(X'^2)\le K$.
Then, for every $k\in\N$, the projection of $\iota(\mathcal G_K)$ to the $k$th coordinate contains only (some of) those compact sets which are bounded by $1+K$ in the $d^{[k]\cup[k]^2}_\infty$-norm (again, cf. the last paragraph of the proof of \Cref{lem:compactness}).
That is, the said projection is a subset of the hyperspace
\begin{equation*}
¨\begin{split}
\mathcal L_{k,K}:=\big\{F\subseteq\R^{[k]\cup[k]^2}:F\text{ is non-empty, compact}\\
\text{ and bounded by }1+K\text{ in the }d^{[k]\cup[k]^2}_\infty\text{-norm}\big\},
\end{split}
\end{equation*}
which is compact and metrizable in the Vietoris topology (see eg.~\cite[Theorem~4.26]{Kechris}).
Consequently, $\iota(\mathcal G_K)$ is a subset of the compact metrizable product space $\prod_{k\in\mathbb N}\mathcal L_{k,K}$.

Fix an arbitrary sequence $(\pi_n',\mu_n')_{n=1}^\infty$ with $[(\pi_n',\mu_n')]\in\mathcal G_K$, $n\in\N$.
Then the sequence $(\iota([(\pi_n',\mu_n')]))_{n=1}^\infty$ has a convergent subsequence in $\prod_{k\in\mathbb N}\mathcal L_{k,K}$.
Since $\mathcal L_{k,K}\subseteq\K_k$ for every $k\in\N$, this means that the corresponding subsequence of $(\pi_n',\mu_n')_{n=1}^\infty$ is convergent in the sense of \Cref{def:convergence}.
By \Cref{thm:existence} the subsequence has a limit $(\pi',\mu')$, for which we obviously have that $[(\pi',\mu')]\in\mathcal G_k$.
This shows that the topology of convergence of $k$-shapes $\S_k(\cdot)$ restricted to $\mathcal G_K$ is compact.

Finally, \Cref{cor:spojitost} easily implies that the identity map on $\mathcal G_K$ is continuous from the topology of convergence of $k$-shapes $\S_k(\cdot)$ to the topology of convergence of $k$-shapes $\widetilde\S_k(\cdot)$.
Combining the continuity of the identity map with the compactness of the former topology, both topologies must coincide on $\mathcal G_K$.
The conclusion follows.
\end{proof}

As a corollary, we provide a proof of a variant of~\cite[Theorem~4.5]{KLS} (we just formulate the result for sequences of s-graphons instead of graph sequences).
We recall that an s-graphon is a symmetric Borel probability measure on the square of the unit interval.
Let $\B_{[0,1]}$ be the Borel $\sigma$-algebra on $[0,1]$, and let $\lambda$ be the Lebesgue measure on the measurable space $([0,1],\B_{[0,1]})$.
We further recall that a sequence $(\mu_n)_{n=1}^\infty$ of s-graphons is s-convergent (and its s-limit is an s-graphon $\mu$) if, for every $k\in\N$, the sequence $\big(\widetilde\S_k(\lambda,\mu_n)\big)_{n=1}^\infty$ is convergent in the Vietoris topology (and its limit is $\widetilde\S_k(\lambda,\mu)$).

\begin{corollary}[\cite{KLS}]
Let $(\mu_n)_{n=1}^\infty$ be an s-convergent sequence of s-graphons.
Then there is an s-graphon $\mu$ which is an s-limit of the sequence $(\mu_n)_{n=1}^\infty$.
\end{corollary}

\begin{proof}
By \Cref{thm:2topologie}, the sequence $(\lambda,\mu_n)_{n=1}^\infty$ of \sqg s on $([0,1],\B_{[0,1]})$ is convergent (in the sense of \Cref{def:convergence}).
By \Cref{thm:existence}, there is a \sqg\ $(\pi,\mu)$ on some measurable space $(X,\A)$, which is a limit of the sequence $(\lambda,\mu_n)_{n=1}^\infty$.
By \Cref{thm:2topologie}, for every $k\in\N$, the sequence $\big(\widetilde\S_k(\lambda,\mu_n)\big)_{n=1}^\infty$ converges to $\widetilde\S_k(\pi,\mu)$.

By \Cref{rem:nonatomic}, we may assume that
\[(X,\A)=([0,1],\B_{[0,1]})\quad\text{and}\quad\pi=\lambda.\]
So it remains to show that the measure $\mu$ on $([0,1],\B_{[0,1]})^2$ is a symmetric probability measure.

The facts that each $\mu_n$ is a probability measure and that the sequence $\big(\widetilde\S_1(\lambda,\mu_n)\big)_{n=1}^\infty$ converges to $\widetilde\S_1(\lambda,\mu)$ easily imply that $\mu$ is a probability measure.

Finally, by \Cref{ex:symmetryPositive}, all $3$-shapes $\S_3(\lambda,\mu_n)$, $n\in\N$, consist only of (some of) those \sqg s whose distributions of edges are symmetric measures.
By convergence, the $3$-shape $\S_3(\lambda,\mu)$ also consists only of (some of) those \sqg s whose distributions of edges are symmetric measures.
So the symmetry of $\mu$ follows by \Cref{ex:symmetry}.
\end{proof}

\section{Final remarks}
\label{sec:finalRemarks}

In the next proposition, we show that isomorphisms between \sqg s with reasonable underlying measurable spaces correspond to isomorphisms between the underlying measurable spaces which preserve the distributions of vertices and edges.

\begin{proposition}
Let $(\pi_X,\mu_X)$ and $(\pi_Y,\mu_Y)$ be \sqg s on measurable
spaces $(X,\A)$ and $(Y,\B)$, respectively.
Let $f\colon X\to Y$ be an isomorphism between the measurable spaces $(X,\A)$ and $(Y,\B)$ such that $f_*\pi_X=\pi_Y$ and $f^{\otimes2}_*\mu_X=\mu_Y$.
Then the Markov kernel $\kappa_f$ is an isomorphism from $(\pi_X,\mu_X)$ to $(\pi_Y,\mu_Y)$ in the category of all \sqg s.

If all singleton subsets of $X$ and $Y$ are measurable, then every isomorphism from $(\pi_X,\mu_X)$ to $(\pi_Y,\mu_Y)$ is of this form.
\end{proposition}

\begin{proof}
Suppose that $f\colon X\to Y$ is an isomorphism between the measurable spaces $(X,\A)$ and $(Y,\B)$ such that $f_*\pi_X=\pi_Y$ and $f^{\otimes2}_*\mu_X=\mu_Y$.
Then $f^{-1}_*\pi_Y=\pi_X$ and $(f^{-1})^{\otimes2}_*\mu_Y=(f^{\otimes2})^{-1}_*\mu_Y=\mu_X$.
To prove that $\kappa_f$ is an isomorphism, we will show that $\kappa_f\circ\kappa_{f^{-1}}=1_{(\pi_Y,\mu_Y)}$ and $\kappa_{f^{-1}}\circ\kappa_f=1_{(\pi_X,\mu_X)}$.
By symmetry, it is enough to show the former equation.
So we fix $y\in Y$ and $B\in\B$ and compute
\begin{equation*}
\begin{split}
\kappa_f\circ\kappa_{f^{-1}}(y,B)=&\int_X\kappa_f(x,B)\kappa_{f^{-1}}(y,dx)\\
=&\kappa_{f^{-1}}\big(y,f^{-1}(B)\big)\\
=&
\begin{cases}
1,&y\in B,\\
0,&y\notin B,
\end{cases}
\end{split}
\end{equation*}
which we needed.

Now suppose that all singleton subsets of $X$ and $Y$ are measurable and that $\kappa$ is an isomorphism from $(\pi_X,\mu_X)$ to $(\pi_Y,\mu_Y)$.
Then there is a morphism $\kappa'$ from $(\pi_Y,\mu_Y)$ to $(\pi_X,\mu_X)$ such that $\kappa\circ\kappa'=1_{(\pi_Y,\mu_Y)}$ and $\kappa'\circ\kappa=1_{(\pi_X,\mu_X)}$.
In particular, for every $x\in X$, it holds that
\begin{equation*}
\begin{split}
1=&1_{(\pi_X,\mu_X)}(x,\{x\})\\
=&\kappa'\circ\kappa(x,\{x\})\\
=&\int_Y\kappa'(y,\{x\})\kappa(x,dy),
\end{split}
\end{equation*}
and so $\kappa'(\cdot,\{x\})=1$ on a set of full $\kappa(x,\cdot)$-measure.
So, for every $x\in X$, there is at least one $y\in Y$ such that $\kappa'(y,\{x\})=1$.
By symmetry, for every $y\in Y$, there is at least one $x\in X$ such that $\kappa(x,\{y\})=1$.

Suppose that there is $x\in X$ such that there are two distinct elements $y_1,y_2\in Y$ with $\kappa'(y_i,\{x\})>0$, $i=1,2$.
Let $x_i\in X$ be such that $\kappa(x_i,\{y_i\})=1$, $i=1,2$.
Then either $x_1\neq x$ or $x_2\neq x$.
Let us assume that $x_1\neq x$ (the other case is analogous).
Then
\begin{equation*}
\begin{split}
0=&1_{(\pi_X,\mu_X)}(x_1,\{x\})\\
=&\kappa'\circ\kappa(x_1,\{x\})\\
=&\int_Y\kappa'(y,\{x\})\kappa(x_1,dy)\\
=&\kappa'(y_1,\{x\})\\
>&0,
\end{split}
\end{equation*}
a contradiction.
It follows that, for every $x\in X$, there is a unique $f(x)\in Y$ such that $\kappa'(f(x),\{x\})>0$, and then $\kappa'(f(x),\{x\})=1$.
By symmetry, for every $y\in Y$, there is a unique $g(y)\in X$ such that $\kappa(g(y),\{y\})>0$, and then $\kappa(g(y),\{y\})=1$.
This defines maps $f\colon X\to Y$ and $g\colon Y\to X$.

For every $x\in X$, it holds that
\begin{equation*}
\begin{split}
1_{(\pi_X,\mu_X)}(g(f(x)),\{x\})=&\kappa'\circ\kappa(g(f(x)),\{x\})\\
=&\int_Y\kappa'(y,\{x\})\kappa(g(f(x)),dy)\\
=&\kappa'(f(x),\{x\})\\
=&1,
\end{split}
\end{equation*}
and so $g\circ f$ is the identity on $X$.
By symmetry, $f\circ g$ is the identity on $Y$.
Altogether, $f$ is a bijection and $g=f^{-1}$.

For every $B\in\B$, we have
\[
f^{-1}(B)=g(B)=\{x\in X:\kappa(x,B)=1\},
\]
so $f$ is measurable by the definition of a Markov kernel.
By symmetry, $g$ is also measurable.
It follows that $f$ is an isomorphism between the measurable spaces $(X,\A)$ and $(Y,\B)$.

Finally, it holds $\kappa=\kappa_f$, as
\[
\kappa\big(x,\{f(x)\}\big)=\kappa\big(g(f(x)),\{f(x)\}\big)=1,\quad x\in X,
\]
and the conclusion follows.
\end{proof}

The next example shows that the situation can be more tricky if singletons are not required to be measurable.

\begin{example}
For every $k\in\N$, let $\A_k$ be the trivial $\sigma$-algebra on the set $[k]$, that is,
\[\A_k=\big\{\emptyset,[k]\big\},\quad k\in\N.\]
Then, for every $k\in\N$, there is only one probability measure on $([k],\A_k)$.
Consequently, for every $k,l\in\N$, there is exactly one Markov kernel from $([k],\A_k)$ to $([l],\A_l)$.
If $(\pi_k,\mu_k)$ and $(\pi_l,\mu_l)$ are \sqg\ on $([k],\A_k)$ and $([l],\A_l)$, respectively, and if $\mu_k([k]^2)=\mu_l([l]^2)$, then the unique Markov kernel from $([k],\A_k)$ to $([l],\A_l)$ is an isomorphism from $(\pi_k,\mu_k)$ to $(\pi_l,\mu_l)$.
\end{example}

Let $C$ be the Cantor set, $\B_C$ be the Borel $\sigma$-algebra on $C$ and $\pi_C$ be the natural probability measure on $(C,\B_C)$.
In~\cite[Lemma~3.1]{KLS}, it was shown that, if $\mu$ is a symmetric finite measure on $(C,\B_C)^2$ which is absolutely continuous with respect to $\pi_C^2$, then
\[
\widetilde\S_k(\pi_C,\mu)=\overline{\big\{\nu:(\gamma_k,\nu)\in(\pi,\mu)^\downarrow_k\big\}}=\big\{\nu:(\gamma_k,\nu)\in(\pi,\mu)^\downarrow_k\big\}.
\]
This suggests that, in some cases, the topological closure from \Cref{def:k-shape} is redundant.
However, in the following example, we will show that it is not the case in general.

\begin{example}
\label{ex:shapeNotClosed}
Let $(q_n)_{n=1}^\infty$ be an enumeration (without repetitions) of all rational numbers from the interval $(0,2)$.
For every $n\in\N$, let $\mu_n$ be the measure on $([0,1],\B_{[0,1]})^2$ which is uniquely given by the following properties:
\begin{itemize}
\item $\mu_n$ is concentrated on the line segment $\{(x,y)\in[0,1]^2:x+y=q_n\}$,
\item the restriction of $\mu_n$ to the above line segment is a multiple of the one-dimensional Lebesgue measure on that segment,
\item $\mu_n([0,1]^2)=\frac1{2^n}$.
\end{itemize}
Let $\mu$ be the measure on $([0,1],\B_{[0,1]})^2$ given by $\mu=\sum_{n=1}^\infty\mu_n$.
Let $\nu$ be the measure on $\F_2^2$ given by
\begin{equation*}
\nu(\{(i,j)\})=
\begin{cases}
0,&i=j,\\
\frac12,&i\neq j,
\end{cases}
\quad i,j\in[2].
\end{equation*}
Then the \sqg\ $(\gamma_2,\nu)$ on $\F_2$ belongs to $\S_2(\lambda,\mu)$, but not to $(\lambda,\mu)^\downarrow_2$.

Towards the proof, we first show that $(\gamma_2,\nu)\notin(\lambda,\mu)^\downarrow_2$.
Suppose for a contradiction that $\kappa$ is a morphism from $(\lambda,\mu)$ to $(\gamma_2,\nu)$.
Then
\begin{equation}
\label{eq:a}
\begin{split}
0=&\nu(\{(1,1)\})\\
=&\kappa^{\otimes 2}_*\mu(\{(1,1)\})\\
=&\int_{[0,1]^2}\kappa(x_1,\{1\})\kappa(x_2,\{1\})\,d\mu(x_1,x_2).
\end{split}
\end{equation}
Let $n_0\in\N$ be such that $q_{n_0}=1$.
Then $\mu_{n_0}$ is concentrated on the line segment $\{(x,y)\in[0,1]^2:x+y=1\}$ and its restriction to this line segment is the $\frac1{2^{n_0+\frac12}}$-multiple of the one-dimensional Lebesgue measure on that segment.
So, \eqref{eq:a} together with the fact that $\mu=\sum_{n=1}^\infty\mu_n$ easily imply that
\begin{equation*}
\begin{split}
0=&\int_{[0,1]^2}\kappa(x_1,\{1\})\kappa(x_2,\{1\})\,d\mu_{n_0}(x_1,x_2)\\
=&\frac1{2^{n_0}}\int_0^1\kappa(x,\{1\})\kappa(1-x,\{1\})\,d\lambda(x).
\end{split}
\end{equation*}
It follows that, for $\lambda$-almost every $x\in[0,1]$, either $\kappa(x,\{1\})=0$ or $\kappa(1-x,\{1\})=0$.
So
\begin{equation}
\label{eq:b}
\lambda\big(\{x\in[0,1]:\kappa(x,\{1\})=0\}\big)\ge\frac12.
\end{equation}
On the other hand, it holds that
\begin{equation}
\label{eq:c}
\int_0^1\kappa(x,\{1\})\,d\lambda(x)=\kappa_*\lambda(\{1\})=\gamma_2(\{1\})=\frac12.
\end{equation}
Combining~\eqref{eq:b} and~\eqref{eq:c} together, we obtain that $\kappa(\cdot,\{1\})$ is, up to a modification on a $\lambda$-null set, the characteristic function of some Borel set $A\subseteq[0,1]$ with $\lambda(A)=1/2$.
By our definition of the measure $\mu$, the function \[(x_1,x_2)\mapsto\kappa(x_1,\{1\})\kappa(x_2,\{1\}),\quad (x_1,x_2)\in[0,1]^2,\]
coincides with the characteristic function of $A^2$ $\mu$-almost everywhere.
Thus, by~\eqref{eq:a}, we must have $\mu(A^2)=0$.

By Lebesgue's density theorem, we can find rational numbers $0<q<q'<1$ such that
\begin{equation}
\label{tyCoJsouVA}
\lambda\big(A\cap[q,q']\big)>\frac12\lambda\big([q,q']\big).
\end{equation}
As the map $x\mapsto q+q'-x$ is an automorphism of the interval $[q,q']$ which preserves the Lebesgue measure, it follows that
\begin{equation}
\label{tyCoSeZobraziDoA}
\lambda\big(\{x\in[q,q']:q+q'-x\in A\}\big)=\lambda\big(\{A\cap[q,q']\})>\frac12\lambda([q,q']\big).
\end{equation}
By~\eqref{tyCoJsouVA} and~\eqref{tyCoSeZobraziDoA}, it holds that
\begin{equation}
\label{eq:LDT3}
\lambda\big(\{x\in[q,q']:x\in A\text{ and }q+q'-x\in A\}\big)>0.
\end{equation}
Let $n_1\in\N$ be such that $q_{n_1}=q+q'$.
Then, by definition of the measure $\mu_{n_1}$, equation~\eqref{eq:LDT3} implies that $\mu_{n_1}(A^2)>0$, which is a contradiction with the proven fact that $\mu(A^2)=0$.

It remains to prove that $(\gamma_2,\nu)\in\S_2(\lambda,\mu)$.
To this end, we must show that $(\gamma_2,\nu)$ can be approximated, with an arbitrary precision, by elements of $(\lambda,\mu)^\downarrow_2$.
So let us fix $\varepsilon>0$ arbitrarily small.
We find $N\in\N$ such that
\begin{equation}
\label{eq:epsilonN}
\frac1{2^{N+1}}<\varepsilon.
\end{equation}
Let $s\in\N$ be such that the rational numbers $q_1,\ldots,q_N$ can be written in the form
\[q_n=\frac{r_n}s,\quad n=1,\ldots,N,\]
for some $r_1,\ldots,r_N\in\N$.
For every $k\in\{1,2,\ldots,2s\}$, we put
\[I_k=\left(\frac{k-1}{2s},\frac k{2s}\right).\]
We define a measurable map $f_\varepsilon\colon[0,1]\to[2]$ by
\[
f_\varepsilon(x)=
\begin{cases}
1,&x\in I_k\text{ for some odd }k\in\{1,2,\ldots,2s\},\\
2,&\text{otherwise}.
\end{cases}
\]
Then it is easy to see that, by our choice of $s$, it holds that
\begin{equation}
\label{eq:sudeLiche}
\mu_n\big(f_\varepsilon^{-1}(\{1\})\times f_\varepsilon^{-1}(\{2\})\big)
=\frac12\mu_n\big([0,1]^2\big)=\frac1{2^{n+1}},\quad n=1,\ldots,N.	
\end{equation}
Obviously, the pushforward $(f_\varepsilon)_*\lambda$ equals $\gamma_2$.
We put $\nu_\varepsilon=(f_\varepsilon)^{\otimes 2}_*\mu$.
Then the Markov kernel $\kappa_{f_\varepsilon}$ is a morphism from $(\lambda,\mu)$ to $(\gamma_2,\nu_\varepsilon)$, and so $(\gamma_2,\nu_\varepsilon)\in(\lambda,\mu)^\downarrow_2$.
We will show that
\begin{equation}
\label{eq:aproxL2easier}
\Big|\nu_\varepsilon(\{(1,2)\})-\frac12\Big|<\varepsilon.
\end{equation}
In a complete analogy, one could also show that
\begin{equation}
\label{eq:aproxL2easier2}
\Big|\nu_\varepsilon(\{(2,1)\})-\frac12\Big|<\varepsilon.
\end{equation}
As $\mu$ is a probability measure, $\nu_\varepsilon$ is a probability measure as well, and so it will follow from~\eqref{eq:aproxL2easier} and~\eqref{eq:aproxL2easier2} that
\begin{equation*}
\big|\nu_\varepsilon(\{(i,j)\})-\nu(\{(i,j)\})\big|<2\varepsilon,\quad i,j\in[2].
\end{equation*}
This is all we need to finish the proof, as $\varepsilon>0$ was chosen arbitrarily small.

To verify~\eqref{eq:aproxL2easier}, we compute
\begin{equation}
\label{eq:skoro}
\begin{split}
&\Big|\nu_\varepsilon(\{(1,2)\})-\frac12\Big|\\
=&\Big|({f_\varepsilon})^{\otimes 2}_*\mu(\{(1,2)\})-\frac12\Big|\\
=&\Big|\mu\left(f_\varepsilon^{-1}(\{1\})\times f_\varepsilon^{-1}(\{2\})\right)-\frac12\Big|\\
\le&\sum_{n=1}^\infty\Big|\mu_n\left(f_\varepsilon^{-1}(\{1\})\times f_\varepsilon^{-1}(\{2\})\right)-\frac1{2^{n+1}}\Big|\\
\stackrel{\eqref{eq:sudeLiche}}{=}&\sum_{n=N+1}^\infty\Big|\mu_n\left(f_\varepsilon^{-1}(\{1\})\times f_\varepsilon^{-1}(\{2\})\right)-\frac1{2^{n+1}}\Big|.
\end{split}
\end{equation}
Applying the fact that $\mu_n([0,1]^2)=\frac1{2^n}$, $n\in\N$, we can continue~\eqref{eq:skoro} by
\begin{equation*}
\left|\nu_\varepsilon(\{(1,2)\})-\frac12\right|
\le\sum_{n=N+1}^\infty\frac1{2^{n+1}}
=\frac1{2^{N+1}}
\stackrel{\eqref{eq:epsilonN}}{<}\varepsilon,
\end{equation*}
as we wanted.
\end{example}

Let $\sim$ be the equivalence relation on the class of all \sqg s given by
\[
(\pi_1,\mu_1)\sim(\pi_2,\mu_2)\quad\Leftrightarrow\quad\forall k\in\N\ \big(\S_k(\pi_1,\mu_1)=\S_k(\pi_2,\mu_2)\big).
\]
It is natural to ask for some simple characterization of the relation $\sim$ (cf. the problem of a characterization
of the isomorphism relation between s-graphons, \cite[p.~36]{KLS}).
It is clear that if $(\pi,\mu)$ and $(\pi',\mu')$ are \sqg s such that there are morphisms between them in both directions, then $(\pi,\mu)^\downarrow_k=(\pi',\mu')^\downarrow_k$, $k\in\N$, and so $(\pi,\mu)\sim(\pi',\mu')$.
However, this is not a characterization, as we show in the next example.

\begin{example}
There exist \sqg s $(\pi,\mu)$ and $(\pi',\mu')$ such that $(\pi,\mu)\sim(\pi',\mu')$ and such that there is no morphism from $(\pi',\mu')$ to $(\pi,\mu)$.

Indeed, by \Cref{ex:shapeNotClosed}, there are \sqg s $(\pi',\mu')$ and $(\rho,\nu)$ such that
\[
(\rho,\nu)\in\S_2(\pi',\mu')\setminus(\pi',\mu')^\downarrow_2.
\]
Consider the constant sequence of \sqg s, where each element of the sequence equals $(\pi',\mu')$.
Of course, this sequence trivially has a limit, but we are interested in the limit constructed by the procedure described in the proof of \Cref{thm:existence}.
Let us recall that, if we follow that proof, we have the freedom to choose quite arbitrary countable dense subsets $D_k$ of $S_k=\S_k(\pi',\mu')$, $k\in\N$, and then we construct the limit \sqg\ $(\pi,\mu)$ in such a way that (among other properties) every element of $D_k$ belongs to $(\pi,\mu)^\downarrow_k$, $k\in\N$.
In particular, we can choose the set $D_2$ such that $(\rho,\nu)\in D_2$.
Then $(\pi,\mu)\sim(\pi',\mu')$ (because $(\pi,\mu)$ is a limit of the constant sequence) and
\begin{equation}
\label{eq:sporMorphismu}
(\rho,\nu)\in(\pi,\mu)^\downarrow_2.
\end{equation}
So if there was a morphism from $(\pi',\mu')$ to $(\pi,\mu)$ then, by~\eqref{eq:sporMorphismu}, there would also be a morphism from $(\pi',\mu')$ to $(\rho,\nu)$, a contradiction.
\end{example}

\begin{center}
	******
\end{center}

\paragraph{Conclusion.}
We have defined a natural category capturing the well-established concepts of graph limits: graphons and, a bit more general, s-graphons. Our framework allowed us proving the convergence and compactness results, at the same time extending the concept of a graphon in the spirit of weighted directed graphs. We expect that our category-theoretic framework will find applications in the classification of graph limits, possibly discovering new relationships between various types of \sqg s.

\paragraph{Acknowledgments.}
The authors thank Jan Swart for pointing out the Kolmogorov Extension Theorem to us.

\bibliographystyle{plain}
\bibliography{References}

\end{document}